\documentclass[noinfoline,12pt]{imsart}
\usepackage{amsmath,amssymb,amsthm,booktabs}
\usepackage{mathrsfs}
\usepackage{amsmath,amssymb,amsthm,enumerate,framed,graphicx,array,color,multirow,mathrsfs,amsrefs}

\usepackage[utf8]{inputenc}
\usepackage[colorlinks,citecolor=blue,urlcolor=blue]{hyperref}

\usepackage{bm}
\usepackage{euscript}
\usepackage{graphicx}

\usepackage{multicol}
\usepackage[usenames,dvipsnames,svgnames,table]{xcolor}

\usepackage{a4wide}

\usepackage{mathrsfs}
\usepackage{euscript}

\startlocaldefs
\numberwithin{equation}{section} \theoremstyle{plain}
\newtheorem{theorem}{Theorem}[section]
\newtheorem{lemma}{Lemma}[section]

\newtheorem{proposition}{Proposition}[section]

\newtheorem{definition}{Definition}
\newtheorem{remark}{Remark}[section]
\endlocaldefs

\allowdisplaybreaks

\allowdisplaybreaks[4]
\numberwithin{equation}{section}

\def\tr{\text{tr}}

\def\la{\left\langle}
\def\ra{\right\rangle}
\def\e{\varepsilon}
\def\R{\mathbb R}
\def\E{\mathbb E}

\begin{document}
\title{Stochastic maximum principle for systems driven by local martingales with spatial parameters}

\runtitle{Stochastic maximum principle}

  \begin{aug}
    \author{\fnms{~ Jian} \snm{Song~}\ead[label=e2]{txjsong@sdu.edu.cn}}
    \and
    \author{\fnms{~ Meng} \snm{Wang}\ead[label=e3]{wangmeng22@mail.sdu.edu.cn}}
    
    \affiliation{Shandong University}
    \runauthor{ J. Song \and  M. Wang}

%
    \address{ Research Center for Mathematics and Interdisciplinary Sciences, Shandong University, Qingdao, Shandong, 266237, China;
School of Mathematics, Shandong University, Jinan, Shandong, 250100, China\\
      \printead{e2}
   }

    \address{School of Mathematics, Shandong University, Jinan, Shandong, 250100, China\\
      \printead{e3}
    }
  \end{aug}



\begin{abstract}
	We consider the stochastic optimal control problem for the dynamical system of the stochastic differential equation driven by a local martingale with a spatial parameter. Assuming the convexity of the  control domain, we obtain the stochastic maximum principle as the necessary condition for an optimal control,  and we also prove its sufficiency under proper conditions. The stochastic linear quadratic  problem in this setting is also discussed.
\end{abstract}

\maketitle
{
	
	\tableofcontents 
}

\section{Introduction}



This  paper concerns the stochastic maximum principle for the dynamical system of the stochastic differential equation (SDE) driven by a local martingale with a spatial parameter. 

On a filtered probability space $(\Omega, \mathcal F, \{\mathcal F_t\}_{t\ge 0}, \mathbf P)$ 	 that satisfies the usual conditions,  
we consider the following  stochastic controlled system 
\begin{equation}\label{state equation}
\left\{ 
\begin{array}{rl}
dx^{u}(t)= & b\big( t,{x^{u}(t),u(t)}\big) dt+M\big(dt,
x^{u}(t),u(t)\big),  \\ 
x^u\left( 0\right) = & x^u_{0},
\end{array}
\right. 
\end{equation}
where $b: [0,T]\times \R^d\times U \times\Omega\to \R^d$ is an $\{\mathcal F_t\}_{t\ge 0}$-adapted process, and \[\Big\{M(t,x,u), t\in[0,T]\Big\}_{(x,u)\in\R^d\times U}\] is a family of $d$-dimensional  local martingales with the parameter $(x,u)\in\R^d\times U\subset \R^{d}\times \R^k$. We assume that the control domain $U$ is a convex subset of $\R^k$. Let
\begin{equation}\label{e:U}
\mathbf U[0,T]=\left\{u:[0,T]\times \Omega \rightarrow U: u\text{ is } \{\mathcal{F}_{t}\}_{t\ge 0}\text{-adapted and } \E\int_{0}^{T}\left\vert u(t)\right\vert ^{2}dt<\infty \right\}       
\end{equation}
denote the set of all admissible controls. 
The cost functional $J(u)$ is given by
\begin{equation}\label{scp}
J\left( u\right) =E\left[\int_0^Tf\big(t,x^u(t),u(t)\big)dt+\Phi \big( {x^{u}(T)}\big)\right],~~ u\in \mathbf U[0,T],
\end{equation}
where $f:[0,T]\times \R^{d}\times U\rightarrow \R$ and $\Phi :\R^{d}\rightarrow \R$ are measurable functions.

For an optimal control $\bar u\in \mathbf U[0,T]$, i.e., a control $\bar u$ satisfying $J(\bar u)=\inf_{u\in \mathbf U[0,T]} J(u)$, let $\bar x=x^{\bar u}$, and we call $(\bar x, \bar u)$  an optimal pair. The goal of this paper is to find the necessary condition which is the so-called stochastic maximum principle for an optimal pair $(\bar x, \bar u)$ for the optimal control problem \eqref{state equation}--\eqref{scp},  and moreover, we shall  prove the sufficiency of the stochastic maximum principle under proper conditions.

SDEs driven by Brownian motion have been extensively studied, in particular, by the celebrated It\^o calculus. The diffusion processes described by SDEs play an important role in the study of stochastic dynamical systems. 	To study various problems concerning SDEs driven by random vector fields (infinite-dimensional random processes),  Kunita \cite{Kunita90}  developed stochastic calculus for semimartingales  with  spatial parameters and 
studied SDEs of the following form{\color{blue}:}
\begin{equation}\label{sde1}
X_t=x_0+\int_0^t F(ds,X_s),
\end{equation}
where $\left\{F(t,x), t\in[0,T] \right\}_{x\in\R^d}$ is a family of continuous semimartingales with the spatial parameter $x\in \R^d$. Note that  \eqref{state equation} is a specific form of \eqref{sde1}.

On the one hand, It\^o's SDE is a special case of \eqref{sde1} if we set
\[F(t,x)=\int_0^tf_0(r,x)dr+\sum\limits_{k=1}^m\int_0^tf_k(r,x)dB_r^k,\]
where $(B^1,\dots, B^m)$ is an $m$-dimensional Brownian motion. On the other hand, if $F(t,x)$ is a C-Brownian motion, i.e., for any partition $0\leq t_0<t_1\cdots<t_n\leq T$ of $[0,T]$,  the increments $F(t_{i+1},x)-F(t_i,x), i=0,1,\cdots,n$ are  independent, 
Kunita \cite{Kunita86} proved that there exist a sequence of independent Brownian motions $\{B^k\}_{k\in\mathbb N}$ and functions $\{f_k\}_{k\in \mathbb N}$ such that
\[F(t,x)=\int_0^tf_0(r,x)dr+\sum\limits_{k=1}^\infty\int_0^tf_k(r,x)dB_r^k.\]
Thus,  Equation \eqref{sde1} can be   viewed formally as an SDE  driven by 	 infinite-dimensional Brownian motion.

Stochastic optimal control problems of dynamical systems driven by  finite-dimensional Brownian motion have been  studied in depth. Here, we briefly mention some literature on  stochastic maximum principles, which is by no means complete.  Bismut \cite{Bismut78} obtained the local maximum principle for stochastic optimal control problems with a convex control set. Peng \cite{Global} obtained the  maximum principle for the general case in which the diffusion coefficient may contain the control variable and the control domain need not be convex. {More recently,  stochastic maximum principles for mean-field control problems were studied in, for instance,  Li \cite{LiJuan2013}, Buckdahn, Li, and Ma \cite{li2016}, Meyer-Brandis, {\O}ksendal, and Zhou \cite{16}, and for stochastic recursive optimal control problems by employing backward stochastic differential equations (BSDEs) in Chen and Epstein \cite{ChenEpstein2002}, Ji and Zhou \cite{JiZhou2006}, Hu \cite{Hu2017}, etc. For stochastic maximum principles in other various situations, we also refer to, for instance,  Ma and Yong \cite{MaJ}, Hu, Ji, and  Xue \cite{Hu2018}, Tang \cite{Tang}, Zhou \cite{ZhouXY}, Wu \cite{Wuzhen2013}, Han, Peng, and Wu \cite{HanPengWu}, Yong and Zhou \cite{YZ99}, and the references therein.

The present paper concerns the optimal control problem \eqref{state equation}--\eqref{scp} driven by a local martingale with a spatial parameter.   One obvious motivation is that, viewing \eqref{state equation} as an SDE driven by infinite-dimensional Brownian motion, it arises naturally when studying financial markets comprising numerous  stocks.  Indeed, optimal control problems for systems governed by infinite-dimensional stochastic evolution equations have been investigated in, for instance, \cite{HuYPeng1990, dumeng2013, fuhrman2013,lu2014general,fabbri2017stochastic}.   Another motivation comes from the study of an illiquid financial market in which the trades of a single large investor can influence market prices. For such a market, Peter and Dietmer \cite{Peter} employed a family of continuous semimartingales $\{P(t, v), t\in[0,T]\}_{v\in\R}$ to model  the price fluctuations of the risky asset given that  the large investor holds a constant stake of $v$ shares in this asset. 

We would also like to point out that the existence and uniqueness of the solution to  \eqref{sde1} were obtained under suitable Lipschitz conditions in Kunita \cite{Kunita90}, and this result was extended in Liang \cite{Liang2007} to  the non-Lipschitz case. Backward doubly SDEs involving martingales with spatial parameters  were studied in  Bally and Matoussi \cite{Matoussi}  and Song, Song, and Zhang \cite{SSZ19}, and 
the solutions were proved therein to be  probabilistic interpretations (nonlinear Feynman-Kac formulas) for the corresponding stochastic partial differential equations.  

We would like to make a few remarks on our work before ending this introduction. In our optimal control problem \eqref{state equation}--\eqref{scp}, we assume that the control domain $U\subset \R^k$ is a convex set, and this enables us to apply the standard variational method to derive the stochastic maximum principle.   A key step of the variational method is to derive the variational equation (see eq. \eqref{variational equation} in Section \ref{sec:vq}) for the generalized SDE (\ref{state equation}), which  involves calculating the derivatives of the local martingale $M$ with respect to the spatial parameters $x$ and $u$. This is the major difference between our problem and the classical case. Further, to obtain the variational equation, we shall employ the stochastic calculus for semimartingales with parameters developed in \cite{Kunita90}. 
 Furthermore, the corresponding adjoint equation (see BSDE \eqref{adjoint} in Section \ref{sec:mp}) contains an extra  martingale which is orthogonal to $M$  to guarantee the existence and uniqueness of the solution. This is because  the BSDE is driven by a general martingale rather than Brownian motion (see El Karoui and Huang \cite{El94}). Despite all these differences,   we can show that the classical stochastic maximum principle is indeed a special case in our setting.

The rest of this paper is organized as follows.  In Section \ref{sec:pre},  we provide some preliminaries on the stochastic calculus for martingales with spatial parameters.   In Section \ref{sec:smp},  we formulate our optimal control problem, derive the stochastic  maximum principle,  and  prove its sufficiency under proper conditions.    Finally in Section \ref{sec:LQ},  we discuss the linear quadratic optimal control problems (LQ problems) in our setting.

Throughout the article, we use $C$ to denote a generic constant which may vary in different places. 

\section{Preliminaries}\label{sec:pre}

In this section, we collect some preliminaries on regularity results and stochastic calculus for local martingales with t spatial parameters. We refer to \cite{Kunita90} for more details.

We recall some conventional notations. Denote by $\R^d$ the  $d$-dimensional real Euclidean space. We use the notation $\partial_x=\left(\frac{\partial}{\partial x_1},\cdots,\frac{\partial}{\partial x_d}\right)$ for $x\in \R^d$. Then for $\Psi:\R^d\rightarrow \R$, $\partial_x \Psi=\left( \frac\partial{\partial {x_j}} \Psi\right)_{1\times d}$ is a row vector,  and for $\Psi:\R^d\rightarrow \R^n$,  $\partial_x \Psi=\left( \frac\partial{\partial {x_j}} \Psi_i\right)_{n\times d}$ is an $n\times d$ matrix.  For two vectors $u, v\in \R^d$,  $\la u,v\ra$ denotes the scalar product of $u$ and $v$, and $\left\vert v \right\vert=\sqrt{\la v, v\ra}$ means the Euclidean norm of $v$. We also use $\la \cdot , \cdot \ra$ to denote the quadratic covariation of two continuous local martingales.  For  $A,B\in\R^{d\times n }$, we denote the scalar product of $M$ and $N$ by $\la M,N\ra=\tr\big[MN^*\big]$(resp., $\left\Vert M\right\Vert=\sqrt{\tr [MM^*]}$), where the superscript $^*$ stands for  the transpose of vectors or matrices.

 \subsection{Regularity of $M(t,x)$ with respect to the spatial parameter $x$ }

In this subsection, we shall recall some results on the differentiability of
continuous local martingales  with respect to the  spatial parameter $x$.  

Let $M:=\{M(t,x), t\in[0,T]\}_{x\in \R^d}$ be a family of local martingales with joint quadratic variation (quadratic covariation)  on the interval $[0,t]$ given by a.s.
\begin{equation}
\la  M(\cdot, x), M(\cdot, y) \ra_t=\int_{0}^{t}q(s,x,y)ds,
\end{equation}\label{e:qv}
where $q(t,x,y)$ is a predictable  process and is called the local characteristic of $M$.

Let $\alpha=(\alpha_1,\ldots,\alpha_d)$ be a  multi-index, and $|\alpha|  =\alpha_1+\cdots+\alpha_d$. Let $d$ and $l$ be  positive integers and $m$ be a nonnegative integer. Denote by $C^m(\mathbb R^{d};\mathbb{R}^{l})$ or simply $C^m$ the set of  $m$-times continuously differentiable functions $f:\mathbb{R}^{d}\rightarrow \R^{l}$.  We use the convention that if $m=0$, $C^0(\R^d;\R^l)$ is just the set $C(\R^d;\R^l)$ of continuous functions. 

Let $K$ be a subset of $\R^{d}$. Denote
\begin{equation*}
\left\Vert f\right\Vert _{m,K}=\sup_{x\in K}\frac{\left\vert
	f(x)\right\vert }{1+\left\vert x\right\vert   }+\sum\limits_{1\leq \left\vert \alpha \right\vert \leq m}\sup_{x\in K}\left\vert D^{\alpha} f(x)\right\vert,
\end{equation*}
where  $D^{\alpha } :=\frac{\partial ^{\left\vert \alpha \right\vert }}{
	(\partial x_{1})^{\alpha _{1}}\ldots (\partial x_{d})^{\alpha _{d}}}$ is the differential operator.
Then $C^{m}$ is a Fr\'echet space  endowed with seminorms $\{\|\cdot\|_{m,K}: K \subset \R^{d} \text{ is compact}\}.$ When $K=\R^{d}$, we also write $\|\cdot\|_{m}:=\|\cdot\|_{m,\R^{d}}$. Here $C_b^{m}$ denotes the set $\{f\in C^{m}: \|f\|_{m}<\infty\}.$

For a constant $\delta\in (0,1] $, let $C^{m,\delta}$ denote the set of functions $f\in C^m$  such that the partial derivatives  $D^\alpha f$ with $\left\vert\alpha\right\vert=m$ are $\delta$-H\"older continuous. Similarly, $C^{m,\delta}$ is a Fr\'echet space under the seminorms, 
\begin{equation*}
\left\Vert f\right\Vert _{m+\delta, K}:=\left\Vert f\right\Vert
_{m,K}+\sum\limits_{\left\vert \alpha \right\vert =m}\sup_{\substack{ x,y\in
		K \\ x\neq y}}\frac{\left\vert D^{\alpha }f(x)-D^{\alpha }f(y)\right\vert }{\left\vert
	x-y\right\vert ^{\delta}},
\end{equation*}
where  $K$ are compact subsets of $\R^{d}$. Clearly $C^{m,0}=C^{m}$.  We also write $\|\cdot\|_{m+\delta, \R^{d}}:=\|\cdot\|_{m+\delta}$, and denote by $C_b^{m,\delta}$ the set $\{f\in C^{m,\delta}: \|f\|_{m+\delta}<\infty\}.$

We say that  a continuous function $f(t,x), (t,x)\in [0,T]\times \R^{d}$ belongs to the class $C^{m,\delta}$ (or $f(t,\cdot)$ is a $C^{m,\delta}$-valued function) if for each fixed $t\in[0,T]$, $f(t, \cdot)$ belongs to $C^{m,\delta}$ and $\int_0^T \|f(t,\cdot)\|_{m+\delta, K}dt<\infty$ for any compact subset $K\subset \R^{d}.$

Similarly, the function space ${\widetilde{C}}^{m}$ consists of all $\R^{l}$-valued functions $g(x,y)$ that are $m$-times differentiable with respect to each $x,y\in \R^{d}$. For $K\subset \R^{d}$, we define
\begin{equation*}
\left\Vert g\right\Vert _{m,K}^{\thicksim }:=\sup\limits_{ x,y\in K}  
\frac{\left\vert g(x,y)\right\vert }{(1+\left\vert
	x\right\vert )(1+\left\vert y\right\vert)}+\sum\limits_{1\leq \left\vert \alpha \right\vert
	\leq m}\sup 
\limits_{ x,y\in K }\left\vert D_{x}^{\alpha
}D_{y}^{\alpha }g(x,y)\right\vert.
\end{equation*}	
Then ${\widetilde{C}}^{m}$ is a Fr\' echet space equipped with the seminorms $\displaystyle\{ \|\cdot\|^{\thicksim }_{m,K}, K\subset \R^{d} \text{ is compact}\}$.  	

For $\delta\in(0,1]$, we define
\[\left\Vert g\right\Vert_{m+\delta,K}^{\thicksim }=\left\Vert g\right\Vert_{m,K}^{\thicksim }+\sum\limits_{\left\vert \alpha \right\vert
	=m}\left\Vert D_{x}^{\alpha
}D_{y}^{\alpha }g\right\Vert _{\delta,K}^{\thicksim },\]
where 
\begin{equation*}
\left\Vert g\right\Vert _{\delta,K }^{\thicksim }=\sup_{\substack{ x,y,x',y'\in K  \\
		x\neq x', y\neq y'}}
\frac{\left\vert g(x,y)-g(x',y)-g(x,y')+g(x',y')\right\vert }{\left\vert x-x'\right\vert ^{\delta	}\left\vert y-y'\right\vert ^{\delta }}.
\end{equation*}
Let  ${\widetilde{C}}^{m,\delta}$ denote the space of functions $g$ such that   $\left\Vert g\right\Vert_{m+\delta,K}^{\thicksim }<\infty$ for any compact subset $K$, and thus   ${\widetilde{C}}^{m,\delta}$ is a Fr\'echet space with the seminorms $\{ \|\cdot\|^{\thicksim}_{m+\delta,K}, K\subset \R^{d} \text{ is compact}\}$. We also have $\widetilde C^{m,0}=\widetilde C^{m}$.

When $K=\R^{d}$, we write $\|\cdot\|^{\thicksim}_{m}:=\|\cdot\|^{\thicksim}_{m,\R^{d}} $ and $\|\cdot\|^{\thicksim}_{m+\delta}:=\|\cdot\|^{\thicksim}_{m+\delta,\R^{d}}$. We also define $\widetilde{C}_b^{m}:=\{g\in \widetilde{C}^m: \|g\|^{\thicksim}_{m}<\infty\}$ and $\widetilde{C}_b^{m,\delta}:=\{g\in \widetilde{C}^m: \|g\|^{\thicksim}_{m+\delta}<\infty\}. $

Consider a random field $\{F(\omega, t,x), t\in[0,T], x\in\R^{d}\}$. If $F(\omega, t,x)$ is $m$-times continuously differentiable with respect to $x$ for almost all $\omega\in\Omega$ and for all $t\in[0,T]$, then it is called a \emph{$C^m$-valued  process}. Furthermore, if  $t\mapsto F(\omega,t, \cdot)$ is a continuous mapping from $[0,T]$ to $C^m$ for almost all $\omega$, then we call it a \emph{continuous $C^m$-process}.  In the same way, one can define \emph{$C^{m,\delta}$-valued process, continuous $C^{m,\delta}$-process, $\widetilde C^{m}$-valued process,  continuous $\widetilde C^{m}$-process, $\widetilde C^{m,\delta}$-valued process,} and  \emph{continuous $\widetilde C^{m,\delta}$-process}.

The following two theorems (Theorem \ref{differentiability} and Theorem \ref{interchange}), which are adopted from  \cite{Kunita90} (Theorem 3.1.2 and Theorem 3.1.3, respectively),  describe the relationship of the spatial regularity between  local martingales and their joint quadratic variations. 

\begin{theorem}\label{differentiability} 
	Let $\{M(t,x), t\in[0,T]\}_{x\in\R^d}$ be a family of continuous local martingales with $M(0,x)\equiv 0$. Assume that the joint quadratic variation $Q(t,x,y)$ has a modification  of a continuous $\widetilde{C}^{m,\delta}$-process for some $m\in\mathbb N$ and $\delta\in(0,1]$. Then $M(t,x) $  has a modification of continuous $C^{m,\varepsilon}$-process for any $\varepsilon <\delta$. Furthermore{\color{blue},} for  any $|\alpha|\le m$,  $\{D^\alpha_x M(t, x),t\in[0,T]\}_{x\in\R^d}$ is a family of continuous local martingales with the joint quadratic variation $D_x^\alpha D_y^
	\alpha Q(t,x,y).$
\end{theorem}

\begin{theorem}\label{interchange}
	Let $\{M(t,x), t\in[0,T]\}_{x\in\R^d} $ and $\{N(t,y), t\in[0,T]\}_{y\in\R^d}$ be continuous local martingales with values in $C^{m,\delta}$ for some $m\geq 0$ and $\delta\in(0,1]$. Then the joint quadratic variation has a modification of a continuous $\widetilde{C}^{m,\varepsilon}$-process for any $\varepsilon<\delta$. Furthermore, the modification satisfies, for  $|\alpha|,|\beta|\leq m$,
	\begin{equation}\label{exchange}
	D^{\alpha}_xD^{\beta}_y\langle M(\cdot,x),N(\cdot,y)\rangle_t=\langle D^{\alpha}_xM(\cdot,x), D^{\beta}_yN(\cdot,y)\rangle_t
	\end{equation}
	for all $t\in[0,T]$.
\end{theorem}

Fix some nonnegative integer $m$ and  $\delta\in (0,1]$, then the local characteristic $q(t,x,y)$ of $M$ is said to belong to \emph{the class $B^{m, \delta}$}, if $q(t, \cdot, \cdot)$ has a modification of a predictable $\widetilde C^{m,\delta}$-valued process with $\int_0^T \|q(t)\|_{m+\delta, K}^{\thicksim}dt<\infty$ a.s. for any compact set $K\subset \R^d$. Furthermore, if $\int_0^T \|q(t)\|_{m+\delta}^{\thicksim}dt<\infty$ a.s., we say that $q(t,x,y)$ belongs to \emph{the class $B_b^{m,\delta}$},  and if $\left\Vert q(t)\right\Vert^{\thicksim}_{m+\delta}\leq c$ holds  for all $t\in[0,T]$ and $\omega\in \Omega$, we say that $q(t,x,y)$ belongs to \emph{the class $B_{ub}^{m,\delta}$}.


\subsection{Stochastic calculus with respect to local martingales with spatial parameters}

Let $\{X_{t}$ $,0\leq t\leq T\}$ be a $\R^d$-valued predictable process such that
\begin{equation}\label{integrable}
\int_{0}^{T}q(s,X_{s},X_{s})ds<\infty \quad a.s.
\end{equation}
Then the generalized It\^o integral  $M_t(X):=\int_{0}^{t}M(ds,X_{s})$ is well defined and is a local martingale. In particular, if the sample paths of $X_t$ are  continuous a.s., the integral can be approximated by Riemann sums:
\begin{equation}
M_{t}(X)=\int_{0}^{t}M(ds,X_{s})=\lim_{\left\vert \triangle
	\right\vert \rightarrow
	0}\sum_{k=0}^{n-1}\big[M(t_{k+1},X_{t_{k}})-M(t_{k},X_{t_{k}})\big],
\end{equation}
where $\Delta$ is a partition of the  interval $[0,T]$ with $|\Delta|$ being the maximum length of all subintervals. 

Let $Y$ be another predictable process satisfying \eqref{integrable}. Then $M_t(Y)$ is also well defined, and the joint quadratic variation of $M_t(X)$ and $M_t(Y)$ is given by
\begin{equation}\label{e:-cov}
\langle  M(X),M(Y)\rangle_t
=\int_{0}^{t}q(s,X_{s},Y_{s})ds\quad a.s.
\end{equation}

\begin{remark}
	Assume $M(t,x)=g(x)W_t$, where $W_t$ is  a standard Brownian motion and $g$ is a measurable function on $\R^d$ such that $\int_0^T |g(X_s)|^2 ds<\infty$ a.s.	 The quadratic variation of $M$ is 
	\[\la M(\cdot, x), M(\cdot, y) \ra_t =g(x) g(y) t\] 
	with the local characteristic $q(t,x,y)=g(x)g(y)$.  The  stochastic integral \[M_t(X)=\int_0^t M(ds, X_s)\] now coincides with the classical It\^o integral $\int_0^t g(X_s) dW_s$. 
\end{remark}

Let $\Big\{M(t,x)=(M^1(t,x),M^2(t,x),\ldots,M^d(t,x)), t\in[0,T]\Big\}_{x\in \mathbb{R}^d}$ be  a family of $d$-dimensional continuous local martingales. Here $M^{i}(t,x), 1\le i \le q$ are  one-dimensional continuous local martingales with joint quadratic variation
\begin{equation}\label{e:jqv}
\left\langle M^i(\cdot, x), M^j(\cdot, y) \right\rangle_t=\int_0^tq_{ij}(s,x,y)ds\quad a.s.
\end{equation}
Denote $q(t,x,y)=\big(q_{ij}(t,x,y),1\leq i,j\leq d\big)$. Then $q(t,x,y)$ is a $d\times d$-matrix-valued process such that $q_{ij}(t,x,y)$=$q_{ji}(t,y,x)$ a.s. for all $x,y\in \mathbb{R}^d, t\in[0,T]$ and $1\leq i,j\leq d$. Therefore, $q(t, x,y)=q^*(t,y,x)$. Moreover, $q(t,x,x)$ is a nonnegative-definite symmetric matrix a.s. for all $(t,x)\in[0,T]\times\R^d$.

We introduce the following set of stochastic processes{\color{blue}:}
\begin{align*}
	&\mathcal{S}^2\big([0,T];\mathbb  {R}^d\big):=\left\{\phi:[0,T]\times\Omega\rightarrow\mathbb{R}^d; \phi\text{  is   predictable, }  \E\left(\sup\limits_{0\leq t\leq T}\left\vert\phi(t)\right\vert^2\right)<\infty\right\}.
\end{align*}

Consider the following SDE
\begin{equation}\label{sde}
\begin{cases}dX_t=b(t,X_t)dt+M(dt,X_t),\quad t\in(0,T], \\
X_0=x_0,
\end{cases}
\end{equation}
where $x_0\in \R^d$ and $b:[0,T]\times\mathbb{R}^d\times\Omega\rightarrow\mathbb{R}^d$ is an adapted  stochastic process.  

\begin{definition}
	We say that $X=(X_t, t\in[0,T])$ adapted to $\{\mathcal F_t\}_{t\geq 0}$ is a solution to \eqref{sde} if  $X$ satisfies the following integral equation
	\[X_t=X_0 +\int_0^t b(s,X_s) ds+\int_0^t M(ds, X_s)\]
	for  $t\in[0,T]$ almost surely. 	
\end{definition} 

By combining Theorem 3.4.1 and Lemma 3.4.3 in \cite{Kunita90}, we obtain the following result. 
\begin{theorem}\label{existence}
	Assume that there exists a positive constant $K$ such that
	\begin{eqnarray*}
		\left\vert b(t,x)-b(t,y)\right\vert&\leq& K\left\vert x-y\right\vert,\\
		\left\vert	b(t,x)\right\vert &\leq& K(1+\left\vert x\right\vert),\\
	\left\Vert q(t,x,x)-2q(t,x,y)+q(t,y,y)\right\Vert & \leq &  K\left\vert x-y\right\vert^2,\\
		\left\Vert q(t,x,y)\right\Vert&\leq& K(1+\left\vert x\right\vert)(1+\left\vert y\right\vert),
	\end{eqnarray*}
	hold for all $x,y\in \mathbb{R}^d$ a.s. Then SDE \eqref{sde} has a unique solution in $\mathcal{S  }^2\big([0,T];\mathbb{R}^d\big)$.
\end{theorem}
\begin{remark}
If we assume $q\in B^{0,1}_{ub}$, then $q$ satisfies the conditions on $q$ in Theorem \ref{existence}. 
\end{remark}

\begin{remark}\label{class1}
	Consider the following classical SDE
	\begin{equation*}
	X_t=x_0+\int_0^tb(s,X_s)ds+ \int_0^t \sigma(s,X_s)dW_s,
	\end{equation*}
	where $b(\cdot, x)$ and $\sigma(\cdot,x)$ are adapted processes for each fixed $x\in \R^d$ taking values in $\mathbb{R}^d$ and $\mathbb{R}^{d\times d}$ respectively, and $W$ is a $d$-dimensional standard Brownian motion. We can write $\int_0^tM(ds,X_s)=\int_0^t \sigma(s,X_s)dW_s$, where $M(t,x)=\int_0^t\sigma(s,x) dW_s$ with the joint quadratic variation
	\[q_{ij}(t,x,y)=\sum\limits_{k=1}^d \sigma_{ik}(t,x)\sigma_{jk}(t,y).\]
	If  we assume  $\sigma$ is uniformly  Lipschitz and linear growth as in the classical setting, then $q(t,x,y)$ satisfies the conditions in Theorem \ref{existence}.
\end{remark}

\section{ Stochastic maximum principle}\label{sec:smp}

In this section, we derive the stochastic maximum principle for the optimal control problem associated with \eqref{state equation}, \eqref{e:U} and \eqref{scp}. 

\subsection{Formulation of the stochastic optimal control problem}\label{sec:optimal control problem}
Recall the  stochastic controlled system 
\eqref{state equation}
\begin{equation*}
\left\{ 
\begin{array}{rl}
dx^{u}(t)= & b\big( t,{x^{u}(t),u(t)}\big) dt+M\big(dt,
x^{u}(t),u(t)\big),  \\ 
x^u\left( 0\right) = & x^u_{0},
\end{array}
\right. 
\end{equation*}
 the set of all admissible controls defined by \eqref{e:U}
\begin{equation*}
\mathbf U[0,T]=\left\{u:[0,T]\times \Omega \rightarrow U: u\text{ is } \{\mathcal{F}_{t}\}_{t\ge 0}\text{-adapted},\E\int_{0}^{T}\left\vert u(t)\right\vert ^{2}dt<\infty \right\} ,    
\end{equation*}
and the cost functional given by \eqref{scp} 
\begin{equation*}
J\left( u\right) =E\left\{\int_0^Tf\big(t,x(t),u(t)\big)dt+\Phi \big( {x^{u}(T)}\big)\right\}.
\end{equation*}

In \eqref{state equation}, $\left\{M(t,x,u)=(M^1(t,x,u),M^2(t,x,u),\ldots,M^d(t,x,u)), t\in[0,T]\right\}_{(x,u)\in \mathbb{R}^d\times \R^k}$ is  a family of $d$-dimensional continuous local martingales, of which the joint quadratic variation is given by 
\begin{equation}\label{e:jqv'}
\langle M^{i}(\cdot,x,u),M^{j}(\cdot,y,v)\rangle_t=\int_0^t q_{ij}(s,x,u,y,v)ds.
\end{equation}

We assume the following conditions.
\begin{itemize}
	\item [(H1)] The functions $b$, $f$, and $\Phi $ are continuous and continuously differentiable in $(x,u)$. Moreover, $b_x$ and $b_u$  are bounded, and there exists a positive constant $K_1$ such that for all $t\in[0,T]$, $(x,u)\in\R^{d+k}$,
	\[\left(|f_x|+|f_u|\right)(t,x,u)+|\Phi_x(x)|\leq K_1\left(1+|x|+|u|\right).\]

\item[(H2)]For all $ (x,u),(y,v)\in \mathbb{R}^{d+k}$, $q(t,x,u,y,v)$ belongs to  $B_{ub}^{1,\delta}(\R^{d+k}\times\R^{d+k},\R^{d\times d})$ for some $\delta\in (0,1]$. It follows that for $x'=(x,u)\in\R^{d+k}$, $y'=(y,v)\in\R^{d+k}$, the partial derivative
$\left\Vert D_{x'}D_{y'} q(t,x,u,y,v)\right\Vert$ is uniformly bounded in $(x',y')$.
\end{itemize}

In particular,  Condition (H2) implies  $q\in B_{ub}^{0,1}$,  i.e., there exist positive constants $K_2$ and $K_3$ such that 
\begin{align*}
	&\left\Vert q(t,x,u,y,v)\right\Vert \leq K_2(1+\left\vert x\right\vert
	+\left\vert u\right\vert )(1+\left\vert y\right\vert +\left\vert
	v\right\vert );\\
	&\left\Vert   q(t,x,u,y,v)-q(t,x',u',y,v)-q(t,x,u, y',v')+q(t,x',u',y',v')\right\Vert\\
	 \leq&
	K_3\left(\left\vert x-x'\right\vert +\left\vert u-u'\right\vert \right)\left(\left\vert y-y'\right\vert +\left\vert v-v'\right\vert \right).
\end{align*}
This (the second inquality) also yields
\begin{equation}\label{e:condition-q2}
	\left\Vert q(t,x,u,x,u)-2q(t,x,u,y,v)+q(t,y,v,y,v)\right\Vert \leq
	2K_3(\left\vert x-y\right\vert^2 +\left\vert u-v\right\vert^2 ).
\end{equation}
Therefore, assuming  (H1) and (H2),  we can apply Theorem \ref{existence}  to  SDE (\ref{state equation}) which consequently has a unique solution $x^u(t)$ with $\E\left(\sup_{0\leq t\leq T}|x^u(t)|^2\right)<\infty$ for each $u\in \mathbf U([0,T])$.

Recall that the goal of the optimal control problem is to minimize the cost functional $J(u)$ over the set of admissible controls $\mathbf U[0,T]$. Suppose $\overline{u}\in \mathbf U[0,T]$ is an optimal control, i.e., 
\[J(\overline u)=\inf\limits_{u\in \mathbf U[0,T]}J(u), \]
and $\overline{x}:=x^{\overline{u}}\in \mathcal{S}^2([0,T];\mathbb  {R}^d)$ is the corresponding solution of the state equation (\ref{state equation}), then  $(\overline{{x}}{,}\overline{{u}})$ is called an optimal pair.  For  $u\in \mathbf U[0,T]$ and $\varepsilon\in[0,1]$, we define
\begin{equation*}
u^{\varepsilon}(t)=\overline{{u}}(t)+\varepsilon\big(u(t)-
\overline{{u}}(t)\big), ~~t\in[0,T]. 
\end{equation*}
Then, clearly $u^\varepsilon$ converges to $\overline{u}$ in $L^2(\Omega\times[0,T])$ as $\varepsilon$ goes to zero.
Recall that the control domain $U$ is convex,  and hence $u^{\varepsilon }$ belongs to  
$\mathbf U[0,T]$ for each $u\in\mathbf U[0,T]$, and we denote  by \[x^{\varepsilon }(t):=x^{u^{\varepsilon}}(t), ~~ t\in[0,T]\] the corresponding unique solution of  (\ref{state equation}) in $\mathcal{S}^2([0,T];\mathbb  {R}^d)$.

\begin{lemma}
	Assume (H1) and (H2). Let \[y^{\varepsilon}(t)=x^{\varepsilon}(t)-\overline{x}(t).\] Then, there exists a positive constant $C$  independent of  $\varepsilon$ such that
	\begin{equation}\label{e:L2bound-y}
	\E\left[\left\vert y^{\varepsilon}(t)\right\vert^2\right]\leq C \varepsilon^2.
	\end{equation}
\end{lemma}
\begin{proof}
	Clearly, $y^{\varepsilon}(t)$ is a semimartingale  of the following form
	\begin{align*}
	y^{\varepsilon}(t)&=\int_0^t\Big[b\big(s,x^{\varepsilon}(s),u^{\varepsilon}(s)\big)-b\big(s,\overline{x}(s),\overline{u}(s)\big)\Big]ds\\&\hspace{1em}+\int_0^t M\big(ds,x^{\varepsilon}(s),u^{\varepsilon}(s)\big)-\int_0^tM\big(ds,\overline{x}(s),\overline{u}(s)\big).
	\end{align*}
	Applying It\^o's formula to $|y^{\varepsilon}(t)|^2$, we have
	\begin{align}
	\left\vert	y^{\varepsilon}(t)\right\vert^2&=2\int_0^t\la y^{\varepsilon}(s), b\big(s,x^{\varepsilon}(s),u^{\varepsilon}(s)\big)-b\big(s,\overline{x}(s),\overline{u}(s)\big)\ra ds\notag\\
	&+2\int_0^t\la y^{\varepsilon}(s),M\big(ds,x^{\varepsilon}(s),u^{\varepsilon}(s)\big)\ra-2\int_0^t\la y^{\varepsilon}(s),M\big(ds,\overline{x}(s),\overline{u}(s)\big)\ra\notag\\
	& +\sum\limits_{i=1}^d\la \int_0^{\cdot}M^i\big(ds,x^{\varepsilon}(s),u^{\varepsilon}(s)\big)-\int_0^{\cdot}M^i\big(ds,\overline{x}(s),\overline{u}(s)\big)\ra_t.\label{e:y2}
	\end{align}
 Here, we shall prove that $\E\int_0^t\la y^{\varepsilon}(s),M\big(ds,x^{\varepsilon}(s),u^{\varepsilon}(s)\big)\ra$  is equal to zero. Since \[M^\e_t:=\int_0^tM\big(ds,x^\varepsilon(s),u^\varepsilon(s)\big)\] is a continuous $\R^d$-valued local martingale and $y^\varepsilon(t)$ is square integrable, the stochastic integral $\int_0^t\la y^{\varepsilon}(s),dM^\e_s\ra$ is a local martingale as well. Then, it remains to show that the local martingale $\int_0^t\la y^{\varepsilon}(s),dM^\e_s\ra$ is also a martingale.
	The Burkholder--Davis--Gundy inequality yields
	\begin{align*}
	&\E\sup\limits_{0\leq t\leq T}\left|\int_0^t\la y^{\varepsilon}(s),M\big(ds,x^{\varepsilon}(s),u^{\varepsilon}(s)\big)\ra\right|\\
	&\leq C\sum\limits_{j=1}^d\E\left(\int_0^T|y_j^{\varepsilon}(t)|^2q_{jj}\big(t,x^{\varepsilon}(t),u^{\varepsilon}(t),x^{\varepsilon}(t),u^{\varepsilon}(t)\big)dt\right)^{\frac 12}\\&\leq C\sum\limits_{j=1}^d\E\left(\sup\limits_{0\leq t\leq T}|y_j^{\varepsilon}(t)|^2+\int_0^T q_{jj}\big(t,x^{\varepsilon}(t),u^{\varepsilon}(t),x^{\varepsilon}(t),u^{\varepsilon}(t)\big)dt\right)\\&<\infty, 
	\end{align*}
	where the last inequality follows from (H2) and the integrability of $y^\varepsilon$, $x^\varepsilon$ and $u^\varepsilon$. Hence
	\[\E\int_0^t\la y^{\varepsilon}(s),M\big(ds,x^{\varepsilon}(s),u^{\varepsilon}(s)\big)\ra=0. \]
	Similarly, we can also show \[\E\int_0^t\la y^{\varepsilon}(s),M\big(ds,\overline{x}(s),\overline{u}(s)\big)\ra=0.\]
	Now taking expectation for both sides of \eqref{e:y2}, we have
	\begin{align}
	\E\left\vert y^{\varepsilon}(t)\right\vert^2&\leq C\E\left(\int_0^t\left\vert y^{\varepsilon}(s)\right\vert^2+\varepsilon^2\left\vert u(s)-\overline{u}(s)\right\vert^2ds\right)\notag\\
	&\quad+\sum\limits_{i=1}^d\E\int_0^t\Big[ q_{ii}\big(s,x^{\varepsilon}(s),u^{\varepsilon}(s),x^{\varepsilon}(s),u^{\varepsilon}(s)\big)\notag\\
	&\quad- 2q_{ii}\big(s,x^{\varepsilon}(s),u^{\varepsilon}(s),\overline{x}(s),\overline{u}(s)\big)+q_{ii}\big(s,\overline{x}(s),\overline{u}(s),\overline{x}(s),\overline{u}(s)\big)\Big]ds\notag\\
	&\leq C\E\left(\int_0^t\left\vert y^{\varepsilon}(s)\right\vert^2+\varepsilon^2\left\vert u(s)-\overline{u}(s)\right\vert^2ds\right),\label{e:gronwall-y}
	\end{align}
	where in the first inequality we use the Lipschitz property of $b$ and the fact that $2|\la x,y\ra |\leq |x|^2+|y|^2$, and the second inequality follows from \eqref{e:condition-q2}.

Finally, the desired result \eqref{e:L2bound-y} follows from	applying Gronwall's inequality to \eqref{e:gronwall-y}.
\end{proof}

\subsection{Variational equation}\label{sec:vq}

Assume Conditions (H1) and  (H2). From Theorem \ref{differentiability},  $M(t,x,u)$ has a modification of  a continuous $C^{1, \delta'}$-local martingale for any $\delta'\in(0,\delta)$.   In particular, the modification, denoted by $M(t,x,u)$ again, is differentiable with respect to $x$ and $u$.  Moreover, the partial derivatives $\partial_xM(t,x,u)$ and $\partial_uM(t,x,u)$ are continuous local martingales. 

 For notational simplicity, throughout the rest of this article, we write \[dM(t)=M(dt):=M\big(dt,\overline x(t),\overline u(t)\big),\] where $M(t)=\int_0^tM\big(ds,\overline x(s),\overline u(s)\big)$ is a continuous local martingale. We also adopt the following notations,
\[b_{x}(t)=b_{x}\big( t,\overline{x}(t),\overline{u}(t)\big), ~~ b_{u}(t)=b_{u}\big( t,\overline{x}(t),\overline{u}(t)\big),\]
\[\partial
_{x}M(dt)=\partial _{x}M\big(dt,\overline{x}(t),\overline{u}(t)\big),~~
\partial
_{u}M(dt)=\partial _{u}M\big(dt,\overline{x}(t),\overline{u}(t)\big),\] 
where 
\begin{equation*}
b_{x}(t)=\left(\partial_{x_j}b^{i}(t)\right)_{d\times d}=\left[ 
\begin{array}{ccc}
b^1_{x_{1}}(t) & \cdots & b^1_{x_{d}}(t) \\ 
\vdots &  & \vdots \\ 
b^d_{x_{1}}(t) & \cdots & b^d_{x_{d}}(t)
\end{array}
\right],
\end{equation*}
The other matrices $b_u(t), \partial_x M(dt),$ and  $ \partial_u M(dt)$ are defined similarly.

Let $\widehat{x}(t)\in\R^d$ be the solution to the following SDE
\begin{equation}\label{variational equation}
\left\{ 
\begin{array}{rl}
d\widehat{x}(t)= & \big( b_{x}(t) \widehat{x}(t)+b_{u}(t)\left( u(t)-\overline{u}(t)\right) \big) dt+\partial _{x}M(dt)
\widehat{x}(t)+\partial _{u}M(dt) \big( {{u(t)}-\overline{u}(t)}\big), \\ 
\widehat{x}(0)= & 0.
\end{array}
\right.
\end{equation}
Here the multiplication used in $\partial _{x}M(dt)
\widehat{x}(t)$ and $\partial _{u}M(dt) \big( u(t)-\overline{u}(t)\big)$ is the matrix multiplication, for instance, 
\[\partial_xM(dt)\widehat{x}(t)=\left(
\begin{array}{c}
\sum\limits_{j=1}^{d} \widehat{x}_j(t)\partial_{x_j}M^{1}(dt)  \\ 
\vdots  \\ 
\sum\limits_{j=1}^{d} \widehat{x}_j(t)\partial_{x_j}M^{d}(dt)
\end{array}
\right).\]

Now, we show that SDE \eqref{variational equation} has a unique solution in $\mathcal S^2([0,T], \R^d)$. If we denote
	\[\tilde b\big(t,\widehat x(t)\big)=b_{x}(t) \widehat{x}(t)+b_{u}(t)\big( u(t)-\overline{u}(t)\big),\] 
	and 
	\[\int_0^t \tilde M\big(ds,\widehat x(s)\big)=\int_0^t\partial _{x}M(ds)
	\widehat{x}(s)+\int_0^t\partial _{u}M(ds)\big( u(s)-\overline{u}(s)\big),\]
	i.e.,
	\[\tilde M(t,x)=\int_0^t\partial _{x}M(ds)
	x+\int_0^t\partial _{u}M(ds)\big( u(s)-\overline{u}(s)\big).\]
	Then, the varational equation (\ref{variational equation}) becomes
	\begin{equation}
	\left\{ 
	\begin{array}{rl}
	d\widehat{x}(t)= & \tilde b(t,\widehat x(t)) dt+\tilde M\big(dt,\widehat x(t)\big),\\ 
	\widehat{x}(0)= & 0,
	\end{array}
	\right.
	\end{equation}
	which has the same form as (\ref{state equation})
with  the local characteristic $\tilde q(t,x,y)$  of $\tilde M$ being
	\begin{align*}
	\big(\tilde {q}(t,x,y)\big)_{ij}&=x^*\left(\frac{\partial^2q_{ij}\big(t,\overline x(t),\overline u(t),\overline x(t),\overline u(t)\big)}{\partial x \partial y}\right) y\\&+x^*\left(\frac{\partial^2 q_{ij}\big(t,\overline x(t),\overline u(t),\overline x(t),\overline u(t)\big)}{\partial x \partial v}\right) \big(u(t)-\overline u(t)\big)\\&+\big(u(t)-\overline u(t)\big)^*\left(\frac{\partial^2 q_{ij}\big(t,\overline x(t),\overline u(t),\overline x(t),\overline u(t)\big)}{\partial u \partial y}\right) y\\&+\big(u(t)-\overline u(t)\big)^*\left(\frac{\partial^2 q_{ij}\big(t,\overline x(t),\overline u(t),\overline x(t),\overline u(t)\big)}{\partial u \partial v}\right) \big(u(t)-\overline u(t)\big).
	\end{align*}
	It can be easily observed that  $\tilde b$ and $\tilde q$ are uniformly Lipschitz continuous and satisfy the following generalized linear growth condition
\[|\tilde{b}(t,x)|\leq C(|a_t|+|x|),\]
	\[\left\Vert\tilde {q}(t,x,y)\right\Vert\leq C(1+|a_t||x|)(1+|a_t||y|),\]
where $\{a_t\}_{t\in [0,T]}$ is an adapted  square integrable process.
	Then, using the same argument as that in Kunita's proof in \cite{Kunita90} for Theorem \ref{existence}  yields that SDE (\ref{variational equation}) has a unique solution with
 \[
	\E\left[\sup\limits_{0\leq t\leq T}\left\vert \widehat{x}(t)\right\vert ^{2}\right]<\infty. \]

We refer to  \eqref{variational equation} as the \emph{variational equation}  along the optimal pair $(\overline{x},\overline{u})$. We will show in Proposition \ref{estimate} that $ \frac{x^{\varepsilon }(t)-\overline{x}(t)}{
	\varepsilon }$ converges to $\widehat{x}(t)$ in $L^2(\Omega)$ as $\e$ goes to $0$. Set
\begin{equation}\label{e:eta-e}
\eta ^{\varepsilon }(t)=\frac{x^{\varepsilon }(t)-\overline{x}(t)}{
	\varepsilon }-\widehat{x}(t).
\end{equation}

\begin{proposition}\label{estimate}
	Under assumptions (H1) and (H2), for any fixed $T>0$, we  have  
	\begin{equation}\label{e:etae0}
	\displaystyle\lim_{\varepsilon \rightarrow 0}\sup_{0\leq t\leq T}\E\left\vert \eta
	^{\varepsilon }(t)\right\vert ^{2}=0.
	\end{equation}
\end{proposition}

\begin{proof}
	From the state equation (\ref{state equation}) and variational equation (\ref{variational equation}),  we have
	\begin{align*}
	\eta ^{\varepsilon }(t) &=\frac{1}{\varepsilon }\left\{\int_{0}^{t}{b}\big( s,x^{\varepsilon }(s),u^{\varepsilon }(s)\big) -b\big(s,\overline{x}(s),\overline{u}(s)\big)ds+\int_{0}^{t}M\big( ds,x^{\varepsilon }(s),u
	^{\varepsilon }(s)\big)\right. \\
	&\hspace{2em}\left.-\int_{0}^{t}M\big (ds,\overline{x}(s),\overline{u}(s)\big)
	- \varepsilon \int_{0}^{t}\big(b_{x}\left( s\right) \widehat{x}(s)+b_{u}\left( s\right) \left( {u(s)}-\overline{u}(s)\right) \big) ds \right.\\
	&\hspace{2em}\left.-\varepsilon \int_{0}^{t}\partial _{x}M\big(ds,\overline{x}(s),\overline{u}
	(s)\big) \widehat{x}(s)-\varepsilon \int_{0}^{t}\partial _{u}M\big(ds,\overline{x}
	(s),\overline{u}(s)\big)\big( u(s)-\overline{u}(s)\big) \right\}.
	\end{align*}
	Denote		
	\begin{equation*}
	\begin{array}{rl}
	A_{\varepsilon }(t)= & \int_{0}^{1}b_{x}\big(t,\overline{x}(t)+\lambda (
	x^{\varepsilon }(t)-\overline{x}(t)),\overline{u}(t)+\lambda \varepsilon (
	u(t)-\overline{u}(t))\big)d\lambda, \\ 
	B_{\varepsilon }(dt)= & \int_{0}^{1}\partial _{x}M\big(dt,\overline{x}(t)
	+\lambda (x^{\varepsilon }(t)-\overline{x}(t)),\overline{u}(t)+\lambda
	\varepsilon (u(t)-\overline{u}(t))\big)d\lambda,\\ 
	C_{\varepsilon }(t)= & \int_{0}^{1}b_{u}\big(t,\overline{x}(t)+\lambda (x^{\varepsilon }(t)-\overline{x}(t)),\overline{u}(t)+\lambda \varepsilon (
	u(t)-\overline{u}(t))\big)d\lambda,
	\\ 
	D_{\varepsilon }(dt)= & \int_{0}^{1}\partial _{u}M\big(dt,\overline{x}(t)
	+\lambda (x^{\varepsilon }(t)-\overline{x}(t)),\overline{u}(t)+\lambda
	\varepsilon (u(t)-\overline{u}(t))\big)d\lambda, \\ 
	\varphi _{\varepsilon }(t)= & \left[ A_{\varepsilon }(t)-b_{x}(t)\right] 
	\widehat{x}(t)+\left[ C_{\varepsilon }(t)-b_{u}(t)\right] (u(t)-\overline{
		u}(t)), \\ 
	\psi _{\varepsilon }(dt)= & \left[ B_{\varepsilon }(dt)-\partial _{x}M(dt)
	\right] \widehat{x}(t)+\left[ D_{\varepsilon }(dt)-\partial _{u}M(dt)\right] (
	u(t)-\overline{u}(t)).
	\end{array}
	\end{equation*}
	
Using the fact that for a continuously differentiable function $f(x,y): \R^d\times \R^k\to \R^d$ and 
	$\alpha\in \R^d, \beta\in \R^k$,   \[\int_0^1  \big(f_x(x+\alpha \lambda, y+\beta \lambda)\alpha+ f_y(x+\alpha \lambda, y+\beta \lambda) \beta \big) d\lambda = f(x+\alpha, y+\beta)-f(x,y),\] we have
	\begin{equation}
	\left\{ \begin{array}{rl}
	d\eta ^{\varepsilon }(t)= & \big[ A_{\varepsilon }(t)\eta ^{\varepsilon}(t)+\varphi _{\varepsilon }(t)\big] dt+\big[ B_{\varepsilon }(dt)\eta
	^{\varepsilon }(t)+\psi _{\varepsilon }(dt)\big], \\ 
	\eta ^{\varepsilon }(0)= & 0. 
	\end{array}
	\right.
	\end{equation}
	Therefore, 
	\begin{align*}
	\E\left\vert \eta ^{\varepsilon }(t)\right\vert ^{2}&=\sum\limits_{i=1}^{d}\E\left\vert \int_{0}^{t}\big( A_{\varepsilon
	}^{i}(s)\eta ^{\varepsilon }(s)+\varphi _{\varepsilon }^{i}(s)\big)
	ds+\int_{0}^{t}\big( B_{\varepsilon }^{i}(ds)\eta ^{\varepsilon }(s)+\psi
	_{\varepsilon }^{i}(ds)\big) \right\vert ^{2} \\
	&\leq \sum\limits_{i=1}^{d}C\E\Bigg(\left\vert \int_{0}^{t}A_{\varepsilon}^{i}(s)\eta ^{\varepsilon }(s)ds\right\vert ^{2}+\left\vert
	\int_{0}^{t}B_{\varepsilon }^{i}(ds)\eta ^{\varepsilon }(s)\right\vert	^{2}\\
	& \qquad\qquad \qquad\qquad \qquad +\left\vert \int_{0}^{t}\varphi _{\varepsilon }^{i}(s)ds\right\vert
	^{2}+\left\vert \int_{0}^{t}\psi _{\varepsilon }^{i}(ds)\right\vert ^{2}
	\Bigg)\\
	&\leq C\left(\E \int_{0}^{T}\left\vert \eta ^{\varepsilon }(s)\right\vert
	^{2}ds+ J_{\varepsilon}(t)\right),
	\end{align*}
	where
	\begin{equation*}
	J_{\varepsilon }(t)=\sum\limits_{i=1}^{d}\E\left( \left\vert \int_{0}^{t}\varphi
	_{\varepsilon }^{i}(s)ds\right\vert ^{2}+\left\vert \int_{0}^{t}\psi
	_{\varepsilon }^{i}(ds)\right\vert ^{2}\right).
	\end{equation*}	
For simplicity of notations, we denote
\begin{align}\label{sh-notation}
	x_{\lambda,\varepsilon}(t)&=\overline x(t)+\lambda\big(x^{\varepsilon}(t)-\overline x(t)\big),\notag\\
	u_{\lambda,\varepsilon}(t)&=\overline u(t)+\lambda\varepsilon\big(u(t)-\overline u(t)\big).
	\end{align}
	Here, the last inequality holds because of the boundedness of $b_{x}$ 	from assumption (H1) and the following estimation{\color{blue}:}
	\begin{align*}
	&\sum\limits_{i=1}^{d}\E\left\vert \int_{0}^{t}B_{\varepsilon }^{i}(ds)
	\eta ^{\varepsilon }(s)\right\vert ^{2}\\ =&\sum\limits_{i=1}^{d}\E\left\vert \sum\limits_{j=1}^{d}
	\int_{0}^{t} \eta
	_{j}^{\varepsilon }(s)\int_{0}^{1}\partial_{x_{j}}M^{i}\big(ds,x_{\lambda ,\varepsilon }(s),u_{\lambda ,\varepsilon }(s)\big)d\lambda  \right\vert ^{2} \\
	\leq&\sum\limits_{i=1}^{d}\sum\limits_{j=1}^{d}C\E\left\vert\int_{0}^{t} \eta _{j}^{\varepsilon }(s)\int_{0}^{1}\partial  _{x_{j}}M^{i}\big(ds,x_{\lambda ,\varepsilon}(s),u_{\lambda ,\varepsilon}(s)\big)d\lambda \right\vert ^{2} \\
	\leq&\sum\limits_{i=1}^{d}\sum\limits_{j=1}^{d}C\E\int_{0}^{T}\left\vert\eta _{j}^{\varepsilon }(s)\right\vert ^{2}d \,\la \int_0^\cdot\int_{0}^{1}\partial
	_{x_{j}}M^{i}\big(dr,x_{\lambda ,\varepsilon }(r),u_{\lambda ,\varepsilon}(r)\big)d\lambda \ra_s \\
	=&\sum\limits_{i=1}^{d}\sum\limits_{j=1}^{d}C\E\int_{0}^{T}\left\vert \eta_{j}^{\varepsilon }(s)\right\vert ^{2}\left(\int_{0}^{1}\int_{0}^{1}\frac{\partial^{2}q_{ii}\big(s,x_{\lambda_1 ,\varepsilon }(s),u_{\lambda_1 ,\varepsilon}(s),x_{\lambda_2 ,\varepsilon }(s),u_{\lambda_2 ,\varepsilon}(s)\big)}{\partial
		x_{j}\partial y_{j}}d\lambda_1 d\lambda_2 \right)ds \\
	\leq& C\E\int_{0}^{T}\left\vert \eta ^{\varepsilon }(s)\right\vert ^{2}ds. 
	\end{align*}
		Clearly,
	\begin{align*}
	\sup_{0\leq t\leq T}\E\left\vert \eta ^{\varepsilon }(t)\right\vert ^{2}&\leq C\left(\E\int_{0}^{T}\left\vert \eta ^{\varepsilon }(s)\right\vert
	^{2}ds +\sup_{0\leq t\leq T}J_{\varepsilon}(t)\right)\\
	&\leq C\left(\int_{0}^{T}\sup_{0\leq r\leq s}\E\left\vert \eta ^{\varepsilon }(r)\right\vert
	^{2}ds +\sup_{0\leq t\leq T}J_{\varepsilon}(t)\right).
	\end{align*}
		From Gronwall's lemma, we can
	obtain 
	\begin{equation}\label{e:etae}
	\sup_{0\leq t\leq T}\E\left\vert \eta ^{\varepsilon }(t)\right\vert ^{2}\leq C e^{CT}
	\left(\sup_{0\leq t\leq T}J_{\varepsilon }(t)\right).
	\end{equation}
Now, to obtain the desired result,  it suffices to show 	that $\sup\limits_{0\leq t\leq T} J_\varepsilon(t) \rightarrow 0$ as $\varepsilon \rightarrow 0$. Note that
	\begin{align}
	\sup\limits_{0\leq t\leq T}J_{\varepsilon}(t)&=\sup\limits_{0\leq t\leq T}\sum\limits_{i=1}^{d}\E\left( \left\vert \int_{0}^{t}\varphi
	_{\varepsilon }^{i}(s)ds\right\vert ^{2}+\left\vert \int_{0}^{t}\psi
	_{\varepsilon }^{i}(ds)\right\vert ^{2}\right)\notag\\
	&\leq \sum\limits_{i=1}^{d}\E\sup\limits_{0\leq t\leq T}\left( \left\vert \int_{0}^{t}\varphi
	_{\varepsilon }^{i}(s)ds\right\vert ^{2}+\left\vert \int_{0}^{t}\psi
	_{\varepsilon }^{i}(ds)\right\vert ^{2}\right).\label{e:Je}
	\end{align}
	For the first term on the right-hand side of \eqref{e:Je}, we have
	\begin{align*}
	&\sum\limits_{i=1}^{d}\E\sup\limits_{0\leq t\leq T}\left\vert \int_{0}^{t}\varphi _{\varepsilon}^{i}(s)ds\right\vert ^{2} \leq C\sum\limits_{i=1}^{d}\E \int_{0}^{T}\left\vert\varphi _{\varepsilon}^{i}(s)\right\vert^2 ds\\
	= &C\sum\limits_{i=1}^{d}\E\int_{0}^{T}\left\vert \big(A_{\varepsilon}^{i}(s)-b_{x}^{i}(s)\big)\widehat x(s)
	+ \big(C_{\varepsilon}^{i}(s)-b_{u}^{i}(s)\big) \big(u(s)-\overline{u}(s)\big)\right\vert^2ds\\
	\leq& C\E\int_{0}^{T} \left(\left\Vert A_{\varepsilon}(s)-b_{x}(s)\right\Vert^2 \left\vert\widehat{x}(s)\right\vert^2 
	+ \left\Vert C_{\varepsilon}(s)-b_{u}(s)\right\Vert^2 \left\vert u(s)-\overline{u}(s)\right\vert^2 \right)ds\\
	\leq& C\E\int_{0}^{T}\int_{0}^{1}\Bigg(\left\Vert b_{x}(s,x_{\lambda ,\varepsilon }(s),u_{\lambda ,\varepsilon }(s))-b_{x}(s)\right\Vert ^{2}\left\vert \widehat{x}(s)\right\vert ^{2}\\
	&\qquad \qquad\qquad  +\left\Vert b_{u}(s,x_{\lambda ,\varepsilon }(s),u_{\lambda ,\varepsilon }(s))-b_{u}(s)\right\Vert ^{2}\left\vert u(s)-\overline{u}(s)\right\vert ^{2}\Bigg) d\lambda ds.
	\end{align*}
	Thus, 	using the dominated convergence theorem, we can conclude 	that
	\begin{equation}\label{e:varphie}
	\lim_{\varepsilon \rightarrow 0}\sup\limits_{0\leq t\leq T}\sum\limits_{i=1}^{d}\E\left\vert
	\int_{0}^{t}\varphi _{\varepsilon }^{i}(s)ds\right\vert ^{2}=0.
	\end{equation}
	For the second term on the right-hand side of \eqref{e:Je}, 
	\begin{align}
	&\sum\limits_{i=1}^{d}\E\sup\limits_{0\leq t\leq T}\left\vert \int_{0}^{t}\psi _{\varepsilon}^{i}(ds)\right\vert ^{2}\notag\\
	=&\sum\limits_{i=1}^{d}\E\sup\limits_{0\leq t\leq T}\left\vert\int_0^t\big[B_{\varepsilon}^{i}(ds)-\partial_{x}M^{i}(ds)\big]\widehat{x}(s)+\big[D_{\varepsilon}^{i}(ds)-\partial_{u}M^{i}(ds)\big]\big(u(s)-\overline{u}(s)\big)\right\vert^2\notag\\
	\leq&\sum\limits_{i=1}^{d}C\E\sup\limits_{0\leq t\leq T}\Bigg(\left\vert\sum\limits_{j=1}^d\int_0^t\widehat{x}_j(s)\big[B_{\varepsilon}^{ij}(ds)-\partial_{x_j}M^{i}(ds)\big]\right\vert^2\notag\\
	&\qquad  \qquad \qquad +
	\left\vert\sum\limits_{l=1}^k\int_0^t\big(u_l(s)-\overline{u}_l(s)\big) \big[D_{\varepsilon}^{il}(ds)-\partial_{u_l}M^{i}(ds)\big]\right\vert^2\Bigg)\notag\\
	\leq&\sum\limits_{i=1}^{d}C\E\sup\limits_{0\leq t\leq T}\Bigg(\sum\limits_{j=1}^d\left\vert\int_0^t\widehat{x}_j(s)\big[B_{\varepsilon}^{ij}(ds)-\partial_{x_j}M^{i}(ds)\big]\right\vert^2\notag\\
	&\qquad  \qquad \qquad +\sum\limits_{l=1}^k\left\vert\int_0^t\big(u_l(s)-\overline{u}_l(s)\big) \big[D_{\varepsilon}^{il}(ds)-\partial_{u_l}M^{i}(ds)\big]\right\vert^2\Bigg)\notag\\
	\leq&C\sum\limits_{i=1}^{d}\E\Bigg(\sum\limits_{j=1}^{d}\int_0^T\left\vert\widehat{x}_j(s)\right\vert^2d\la\int_0^\cdot\int_0^1\partial_{x_j}M^i\big(dr,x_{\lambda,\varepsilon}(r),u_{\lambda,\varepsilon}(r)\big)d\lambda-\int_0^\cdot\partial_{x_j}M^{i}(dr)\ra_s\notag\\&+\sum\limits_{l=1}^k\int_0^T\left\vert u_l(s)-\overline{u}_l(s)\right\vert^2d\la\int_0^\cdot\int_0^1\partial_{u_l}M^i\big(dr,x_{\lambda,\varepsilon}(r),u_{\lambda,\varepsilon}(r)\big)d\lambda-\int_0^\cdot\partial_{u_l}M^{i}(dr)\ra_s\Bigg). \label{e:psie}
	\end{align}
	
	Note that
	\begin{align*}
	&\la \int_0^\cdot\int_0^1\partial_{x_j}M^i\big(dr,x_{\lambda,\varepsilon}(r),u_{\lambda,\varepsilon}(r)\big)d\lambda-\int_0^\cdot\partial_{x_j}M^{i}(dr)\ra_s \\=&\int_0^s\left(\int_{0}^{1}\int_{0}^{1}\frac{\partial^{2}q_{ii}\big(s,x_{\lambda_{1},\varepsilon }(r),u_{\lambda_{1},\varepsilon }(r),x_{\lambda_{2},\varepsilon }(r),u_{\lambda_{2},\varepsilon }(r)\big)}{
		\partial x_{j}\partial y_{j}}d\lambda _{1}d\lambda _{2}\right.\\
	&\hspace{2em}\left. +\frac{\partial
		^{2}q_{ii}\big(r,\overline{x}(r),\overline{u}(r),\overline{x}(r),\overline{u}
		(r)\big)}{\partial x_{j}\partial y_{j}}-2\int_{0}^{1}\frac{\partial^{2}q_{ii}\big(r,x_{\lambda,\varepsilon }(r),u_{\lambda,\varepsilon }(r),{\overline{x}(r)},\overline{u}(r)\big)}{\partial x_{j}\partial y_{j}}d\lambda \right)dr.
	\end{align*}
	 Recall that in (H2){\color{blue},} we assume $q\in B^{1,\delta}_{ub}$ which yields that  the  partial derivatives $\frac{\partial^2 q}{\partial x_i \partial y_j}$ of $q$ are uniformly bounded. Thus, we have 
	\begin{align}\label{1}
	\sum\limits_{i=1}^{d}\sum\limits_{j=1}^{d}\E\int_0^T\left\vert\widehat{x}_j(s)\right\vert^2d\la\int_0^\cdot\int_0^1\partial_{x_j}M^i\big(dr,x_{\lambda,\varepsilon}(r),u_{\lambda,\varepsilon}(r)\big)d\lambda-\int_0^\cdot\partial_{x_j}M^{i}(dr)\ra_s
	\end{align}is finite.  Furthermore,  (H2) 	implies the continuity of $\frac{\partial^2 q}{\partial x_i \partial y_j}$, and hence \eqref{1} converges to 0 as $\varepsilon \to 0$. The same analysis can be  applied to
	\begin{align*}
	\sum\limits_{i=1}^{d}\sum\limits_{l=1}^k\E\int_0^T\left\vert u_l(s)-\overline{u}_l(s)\right\vert^2d\la\int_0^\cdot\int_0^1\partial_{u_l}M^i\big(dr,x_{\lambda,\varepsilon}(r),u_{\lambda,\varepsilon}(r)\big)d\lambda-\int_0^\cdot\partial_{u_l}M^{i}(dr)\ra_s.
	\end{align*} 
		Then, using the dominated convergence theorem, we have
	\begin{equation*}
	\lim_{\varepsilon \rightarrow 0}\sum\limits_{i=1}^{d}
	\E\left(\sup\limits_{0\leq t\leq T}\left\vert
	\int_{0}^{t}\psi _{\varepsilon }^{i}(ds)\right\vert^2\right) =0.
	\end{equation*}
	The proof is complete.
\end{proof}

\begin{theorem}\label{derivatives}
	Assume (H1) and (H2). Then we have 
	\begin{equation*}
	\lim_{\varepsilon \rightarrow 0}\frac{J(u^{\varepsilon })-J(\overline{u})}{\varepsilon }=\E\left\{\int_0^T\left[f_x(t)\widehat{x}(t)+f_u(t)\big(u(t)-\overline{u}(t)\big)\right]dt+\Phi _x\big({\overline{x}}(T)\big) \widehat{x}(T)\right\}.
	\end{equation*}
	
\end{theorem}
\begin{proof}
	Denote
	\begin{align*}
	H_\e&=
	\frac{1}{\varepsilon}\left(\int_0^T\left[f\big(t,x^{\varepsilon}(t),u^{\varepsilon}(t)\big)-f(t)\right]dt+\Phi\big(x^{\varepsilon}(T)\big)-\Phi\big(\overline x(T)\big)\right)\\&\hspace{1em}-\left(\int_0^T\left[f_x(t)\widehat{x}(t)+f_u(t)\big(u(t)-\overline{u}(t)\big)\right]dt+\Phi _x\big(\overline{x}(T)\big) \widehat{x}(T)\right).
	\end{align*}
	Then, to prove the desired result, it suffices to show 
	$\lim\limits_{\varepsilon \rightarrow 0}\E\big[\left\vert H_\e\right\vert\big]=0$. 
	
		From the Taylor expansion, we have, recalling the definition \eqref{e:eta-e} of $\eta^\e(t)$ and using the abbreviated notations (\ref{sh-notation}) in the last Proposition, 
	
\begin{align*}
H_\e&=\left(\int_0^1\Phi_x\big(x_{\lambda,\varepsilon}(T)\big)d\lambda\right)\eta^{\varepsilon}(T)+\left(\int_0^1\left[\Phi_x\big(x_{\lambda,\varepsilon}(T)\big)-\Phi_x\big(\overline{x}(T)\big)\right]d\lambda\right)\widehat{x}(T)\\&\quad+\int_0^T\left( \int_0^1f_x\big(t,x_{\lambda,\varepsilon}(t), u_{\lambda,\varepsilon}(t)\big)d\lambda\right)\eta^{\varepsilon}(t)dt\\
&\quad+\int_0^T\left(\int_0^1\left[f_x\big(t,x_{\lambda,\varepsilon}(t), u_{\lambda,\varepsilon}(t)\big)-f_x(t)\right]d\lambda\right)\widehat{x}(t)dt\\&\quad+\int_0^T\left(\int_0^1\left[f_u\big(t,x_{\lambda,\varepsilon}(t),u_{\lambda,\varepsilon}(t)\big)-f_u(t)\right]d\lambda\right)\big(u(t)-\overline{u}(t)\big)dt.
	\end{align*}
	Then, the H\"older inequality implies 
	\begin{align*}
	\E\big[\left\vert H_\e\right\vert\big] & \leq \left(\E\left\vert\int_0^1\Phi_x\big(x_{\lambda,\varepsilon}(T)\big)d\lambda\right\vert^2\right)^\frac{1}{2}\left(\E\left\vert\eta^{\varepsilon}(T)\right\vert^2\right)^\frac{1}{2}\\
	&\quad+\left(\E\left\vert \int_0^1\left[\Phi_x\big(x_{\lambda,\varepsilon}(T)\big)-\Phi_x\big(\overline{x}(T)\big)\right]d\lambda\right\vert^2\right)^\frac{1}{2}\left(\E\left\vert\widehat{x}(T)\right\vert^2\right)^\frac{1}{2}\\
	&\quad+\int_0^T\left(\E\left\vert\int_0^1f_x\big(t,x_{\lambda,\varepsilon}(t),u_{\lambda,\varepsilon}(t)\big)d\lambda\right\vert^2\right)^\frac{1}{2}\left(\E\left\vert\eta^{\varepsilon}(t)\right\vert^2\right)^\frac{1}{2}dt\\
	&\quad+\int_0^T\left(\E\left\vert\int_0^1\left[f_x\big(t,x_{\lambda,\varepsilon}(t),u_{\lambda,\varepsilon}(t)\big)-f_x(t)\right]d\lambda\right\vert^2\right)^\frac{1}{2}\left(\E\left\vert\widehat{x}(t)\right\vert^2\right)^\frac{1}{2}dt\\
	&\quad+\int_0^T\left(\E\left\vert\int_0^1\left[f_u\big(t,x_{\lambda,\varepsilon}(t),u_{\lambda,\varepsilon}(t)\big)-f_u(t)\right]d\lambda\right\vert^2\right)^\frac{1}{2}\left(\E\left\vert u(t)-\overline{u}(t)\right\vert^2\right)^\frac{1}{2}dt.
	\end{align*}
Noting Proposition {\ref{estimate}} and that the functions $\Phi_x, f_x$ and  $f_u$ are continuous and satisfy the linear  growth condition,  	we can conclude 	that $\lim\limits_{\varepsilon\rightarrow 0}\E\big[\left\vert H_\e\right\vert\big]=0$ by the dominated convergence theorem.
\end{proof}

\subsection{Maximum principle} \label{sec:mp}

Denote $q(t,x,u, y,v):=(q_{ij}(x,u,y,v))_{d\times d}$ where $q_{ij}$ is given by \eqref{e:jqv'}. Thus we have $q(t,x,u, x',u')=q^*(t,x',u', x,u)$.  
Throughout the rest of this article, we consider both $q:=q(t, x,u, y,v)$ and $q^*:=q^*(t, x,u, y,v)$ as functions of $(t, x,u, y,v)$, and we shall use $\frac{\partial }{\partial x}, \frac{\partial }{\partial u}, \frac{\partial }{\partial y}$ and $\frac{\partial }{\partial v}$ to denote the partial derivatives with respect to $x,u,y$ and $v$, respectively.  Clearly, at any point $p_0=(t_0, x_0, u_0, x_0, u_0)$, we have 
\begin{equation}\label{e:q=q*}
\frac{\partial}{\partial x} q^*(p_0) =\frac{\partial }{\partial y} q(p_0), ~~ \frac{\partial}{\partial u} q^*(p_0) =\frac{\partial }{\partial v} q(p_0).
\end{equation}

Now we consider the adjoint equation,	i.e.,  the following BSDE   
\begin{equation}\label{adjoint}
\begin{cases}
dy(t) =  -\Big(b^*_{x}\left( t\right) y(t)+\left(\frac{\partial}{\partial x}\tr\left[z(t) q^*\big(t,\overline{x}(t),\overline{u}(t),\overline{x}(t),\overline{u}(t)\big)\right]\right)^*+f^*_x(t)\Big)dt\\\qquad \qquad \quad +z(t)dM(t)+dN(t), \\ 
y(T) =  \Phi^*_x\big(\overline{x}(T)\big).
\end{cases}
\end{equation}
	Note that as mentioned in Section \ref{sec:optimal control problem},  $dM(t)=M\big(dt,\overline{x}(t),\overline{u}(t)\big)$ and $(\overline x,\overline u)\in \R^{d+k}$ is an optimal pair for the control problem. 

Denote
\begin{align*}
\mathcal{M}^2\big([0,T];\mathbb{R}^d\big):=\Bigg\{\phi:[0,T]\times\Omega\rightarrow\mathbb{R}^d; \,\,\phi\text{ is predictable with } \E\int_0^T \left\vert\phi(t)\right\vert^2dt<\infty\Bigg\},
\end{align*}
and 
\begin{align*}
	\mathcal{Q}^2\big([0,T];\mathbb{R}^{d\times d}\big):=&\Bigg\{\phi:[0,T]\times\Omega\rightarrow\mathbb{R}^{d\times d};\,\,  \phi\text{ is predictable with }  \\
	&\qquad\qquad  \E\int_0^T \tr\Big[\phi(t)q\big(t,\overline x(t),\overline u(t),\overline x(t),\overline u(t)\big)\phi^*(t) \Big]dt<\infty\Bigg\}.
\end{align*}

	Then,
according to \cite{El94}, there exists a unique  triple of stochastic processes  \[\left(y,z,N\right)\in \mathcal{M}^2\big([0,T];\mathbb{R}^d\big)\times\mathcal{Q}^2\big([0,T];\mathbb{R}^{d\times d}\big)\times\mathcal L^2\] satisfying (\ref{adjoint}),  where 
 $\mathcal L^2$ is the space 	containing all square integrable martingales. Here, $N$ is a $\mathbb{R}^d$-valued square integrable martingale orthogonal to  $M$, i.e., for $1\leq i,j \leq d$,  \[\la N^i, \int_0^\cdot M^{j}\big(ds,\overline{x}(s),\overline{u}(s)\big)\ra_t=0, ~~\forall t\in[0,T].\]

\begin{lemma}\label{variational inequality}
	Let $\big(y,z, N\big)$ be the adapted solution of \eqref{adjoint}. Then	
	\begin{align*}
	&\E\la y(T),\widehat{x}(T)\ra\\
=&\E\int_{0}^{T}\left[\la
	b^*_{u}(t)y(t)+\left(\frac{\partial }{\partial u} \tr \left[ z(t) q^*\big(t,\overline{x}(t),\overline{u}(t),\overline{x}(t),\overline{u}(t)\big)\right]\right)^*,u(t)-\overline{u}(t)\ra-\la f_x^*(t),\widehat{x}(t)\ra \right]dt.
	\end{align*}
	
\end{lemma}
\begin{proof}
	Applying It\^o formula to $\la y(t),\widehat{x}(t)\ra$,  we have
	\begin{align*}
	&\E\la y(T),\widehat{x}(T)\ra  \\
	=&\E\int_{0}^{T}\la dy(t),\widehat{x}(t)\ra +\la y(t),d\widehat{x}(t)\ra+d\la y,\widehat{x}\ra_t \\
	=&\E\Bigg[-\int_{0}^{T}\
	\la b^*_{x}(t) y(t)+\left(\frac{\partial}{\partial x}\tr\left[z(t) q^*\big(t,\overline{x}(t),\overline{u}(t),\overline{x}(t),\overline{u}(t)\big)\right]\right)^*+f^*_x(t) ,\widehat{x}(t)\ra dt\\
	&\hspace{3em}+\la y(t),b_{x}(t) \widehat{x}(t)+b_{u}(t) \big(u(t)-\overline{u}(t)\big) \ra dt\\
	&\hspace{3em}+ d\la \int_0^\cdot z(s)dM(s)+N(\cdot),\int_0^\cdot \partial_{x}M(ds)\widehat{x}(s)+\int_0^\cdot \partial _{u}M(ds)\big(u(s)-\overline{u}(s)\big) \ra_t\Bigg],
	\end{align*}
	where we use the notation, for $d$-dimensional local martingales $M=(M^1, \dots, M^d)$ and $N=(N^1, \dots, N^d)$,  \[\la (M^1, \dots, M^d), (N^1,\dots, N^d)\ra_t:=\sum_{j=1}^d \la M^j, N^j\ra_t. \]
	 Note that  from Theorem \ref{interchange} and \eqref{e:q=q*}, 	it follows that
	  \begin{align*}
	  &d \la \int_0^\cdot z(s)M(ds), \int_0^\cdot \partial_xM(ds) \widehat{x}(s) \ra_t\\
	  =&\la\left(\frac{\partial}{\partial y}\tr\left[z(t) 
	 q\big(t,\overline{x}(t),\overline{u}(t),\overline{x}(t),\overline{u}(t)
	 \big)\right]\right)^*,\widehat{x}(t)\ra dt\\
	  =&\la\left(\frac{\partial}{\partial x}\tr\left[z(t) 
	 q^*\big(t,\overline{x}(t),\overline{u}(t),\overline{x}(t),\overline{u}(t)
	 \big)\right]\right)^*,\widehat{x}(t)\ra dt
	 \end{align*}
	 and 
	 \begin{align*}
	 &d \la \int_0^\cdot z(s)M(ds), \int_0^\cdot \partial_uM(ds) \big(u(s)-\overline{u}(s)\big)\ra_t\\
	 =&
	\la \left(\frac{\partial}{\partial v}\tr\left[z(t) q\big(t,\overline{x}(t),\overline{u}(t),\overline{x}(t),\overline{u}(t)\big)\right]\right)^*, u(t)-\overline{u}(t)\ra dt\\
	=&
	\la \left(\frac{\partial}{\partial u}\tr\left[z(t) q^*\big(t,\overline{x}(t),\overline{u}(t),\overline{x}(t),\overline{u}(t)\big)\right]\right)^*, u(t)-\overline{u}(t)\ra dt. 
	\end{align*}
	From the orthogonality of $M$ and $N$, we also have   \[d\la N,\int_0^\cdot \partial_{x}M(ds)\widehat{x}(s)+\int_0^\cdot \partial _{u}M(ds)\big(u(s)-\overline{u}(s)\big)\ra_t =0.\] 
	 
	By combining the above equalities, the desired result can be obtained.
\end{proof}
Now, 	from Theorem \ref{derivatives}, the adjoint equation \eqref{adjoint} and Lemma \ref{variational inequality}, we have 
\begin{align*}
&\lim_{\varepsilon \rightarrow 0}\frac{J(u^{\varepsilon })-J(\overline{u})}{\varepsilon }\\ =&\E\int_{0}^{T}\la
b^*_{u}(t)y(t)+\left(\frac{\partial }{\partial u}\tr\left[ z(t) q^*\big(t,\overline{x}(t),\overline{u}(t),\overline{x}(t),\overline{u}(t)\big)\right]\right)^*+f^*_u(t),u(t)-\overline{u}(t)\ra dt.
\end{align*}                       
Since $\overline{u}$ is an optimal control at which $J(u)$ is minimized, we have for almost all $t\in[0,T]$,
\begin{equation}\label{e:variational}
\la
b^*_{u}(t)y(t)+\left(\frac{\partial }{\partial u}\tr\left[ z(t) q^*\big(t,\overline{x}(t),\overline{u}(t),\overline{x}(t),\overline{u}(t)\big)\right]\right)^*+f^*_u(t),u(t)-\overline{u}(t)\ra \geq 0\quad a.s. 
\end{equation}

We now state our maximum principle in the following theorem, where 	we use the Hamiltonian 	defined as follows:
	\begin{equation}\label{e:hamiltonian}
		H(t,x,u,y,z)=\la y(t),b(t,x,u)\ra +\tr[z(t) q^*(t,\overline x,\overline u,x,u)]+f(t,x,u).
	\end{equation} 
\begin{theorem}\label{mp}
	Assume conditions  (H1)-(H2). Let $\overline{u}$ be an optimal control associated with the stochastic control problem \eqref{state equation}--\eqref{scp} and $(\overline{x},\overline{u})$ be the optimal pair. 	Then, there exists $(y,z)\in\mathcal{M}^2([0,T];\mathbb{R}^d)\times\mathcal{Q}^2([0,T];\mathbb{R}^{d\times d})$ satisfying the adjoint equation \eqref{adjoint} such that for all $u\in \mathbf U[0,T]$,
	\begin{equation}\label{ine}
 H_u\big(t,\overline{x}(t),\overline{u}(t),y(t),z(t)\big)\big(u(t)-\overline{u}(t)\big) \geq 0 \quad a.s.
	\end{equation}
 for almost all $t\in[0,T]$, 	here $H$ is given by \eqref{e:hamiltonian} and $H_u:=\frac{\partial}{\partial u}H$.
\end{theorem}

\begin{remark} We would like to remind the reader that the proof of Theorem \ref{mp} 	heavily relies  on the convexity of  the control domain $U$. For control problems with a general control domain which is not necessarily convex, we suspect that one may still obtain a global maximum principle as in Peng \cite{Global} by applying the second-order variational method developed therein.    Noting that in our setting the noise $M$ is a local martingale with a spatial parameter rather than Brownian motion, we expect that  extra difficulties would come from deriving relevant estimations and adjoint equations as well. This is a subsequent project that we plan to work on in future.
	\end{remark}

\begin{remark}\label{remark:3-2}
If the control domain $U$ is the whole space $\R^k$,  let $\widetilde{u}(t)=-u(t)+2\overline{u}(t)$ for $t\in [0,T]$, and 	then, $\widetilde u\in \mathbf U[0,T]=\mathcal{M}^2([0,T];\mathbb{R}^k)$.  	Now, Theorem \ref{mp} yields
	\begin{equation*}
	H_u\big(t,\overline{x}(t),\overline{u}(t),y(t),z(t)\big)\big(\widetilde u(t)-\overline{u}(t)\big)\geq 0 \quad a.s.,
	\end{equation*}
		i.e.,
	\begin{equation*}
	H_u\big(t,\overline{x}(t),\overline{u}(t),y(t),z(t)\big)\big(u(t)-\overline{u}(t)\big)\leq 0 \quad a.s.
	\end{equation*}
	This implies 
	\[H_u\big(t,\overline{x}(t),\overline{u}(t),y(t),z(t)\big)=0\quad a.s.\]
\end{remark}

\begin{remark}\label{classical2}
If we assume
	\begin{equation}\label{rel} 		 M(t,x,u)=\int_0^t\sigma(s,x,u)dW_s,
	\end{equation}
	similar to Remark \ref{class1}, the joint quadratic variation of $M(\cdot,x,u)$ and $M(\cdot ,y,v)$ is given by
	\[q(t,x,u,y,v)=\sigma(t,x,u)\sigma^*(t,y,v),\]
	and, the controlled system \eqref{state equation} is reduced to  the classical one:
	\begin{equation}\label{classical-2}
		x^u(t)=x^u_0+\int_0^tb\big(s,x^u(s),u(s)\big)ds+ \int_0^t \sigma\big(s ,x^u(s),u(s)\big)dW_s,
	\end{equation}
	The adjoint equation \eqref{adjoint}  becomes 
	\begin{equation}\label{e:adjoint'}
		\begin{cases}
			dy(t)= & -\Big(b^*_{x}\left( t\right) y(t)+\left(\frac{\partial}{\partial x}\tr\left[z(t)\sigma \big(t,\overline{x}(t),\overline{u}(t)\big)\sigma^*\big(t,\overline{x}(t),\overline{u}(t)\big)\right]\right)^*+f^*_x(t)\Big)dt\\&+z(t)\sigma \big(t,\overline{x}(t),\overline{u}(t)\big)dW_t+ dN(t), \\ 
			y(T)= & \Phi^*_x\big(\overline{x}(T)\big),
		\end{cases}
	\end{equation}
	where 
	\begin{align*}&\frac{\partial }{\partial x} \tr\left[z(t)\sigma \big(t,\overline{x}(t),\overline{u}(t)\big)\sigma^*\big(t,\overline{x}(t),\overline{u}(t)\big)\right]\\
	:= &\frac{\partial }{\partial x} \tr\left[z(t)\sigma \big(t,y,v\big)\sigma^*\big(t,x,u\big)\right]\Big|_{(x,u, y,v)=(\overline{x}(t),\overline{u}(t),\overline{x}(t), \overline{u}(t))}\,.
	\end{align*}
If we assume 	that the filtration is generated by the Brownian motion $W$, then a mean-zero local martingale $N$ is orthogonal to $W$ if and only if $N\equiv0$. Denoting $\tilde{z}(t)=z(t)\sigma\big(t,\overline x(t),\overline u(t)\big)$, the adjoint equation \eqref{e:adjoint'} can be written as 
	\begin{align*}
		\left\{ 
		\begin{array}{rl}
			dy(t)= & -\Big(b^*_{x}\left( t\right) y(t)+\left(\frac{\partial}{\partial x}\tr\left[\tilde z(t)\sigma^*\big(t,\overline{x}(t),\overline{u}(t)\big)\right]\right)^*+f^*_x(t)\Big)dt+\tilde z(t)dW_t, \\ 
			y(T)= & \Phi^{*}_x\big(\overline{x}(T)\big),
		\end{array}
		\right. 
	\end{align*}
	and the variational inequality \eqref{e:variational} becomes
	\begin{equation}\label{va-c}
	\la
	b^*_{u}(t)y(t)+\left(\frac{\partial }{\partial u}\tr\left[ \tilde  z(t)\sigma^* \big(t,\overline{x}(t),\overline{u}(t) \big)\right]\right)^*+f^*_u(t),u(t)-\overline{u}(t)\ra\geq 0\quad a.s.
	\end{equation} 
from which  the classical maximum principle can be 	obtained.
\end{remark}

	\subsection{Sufficiency of 	the maximum principle}
	
	In this subsection, we show that the necessary condition \eqref{ine} for an optimal control pair obtained in Theorem \ref{mp}  is also sufficient under proper conditions. 
	\begin{theorem}
		Suppose (H1)-(H2) hold. Let $\overline u\in \mathbf U[0,T] $ satisfy that,  for all $u\in \mathbf U[0,T]$,
		\begin{equation}
			H_u\big(t,\overline{x}(t),\overline{u}(t),y(t),z(t)\big)\big(u(t)-\overline{u}(t)\big)\geq 0 \quad a.s., 
		\end{equation}
 for almost all $t\in[0,T]$, where $H$ is given in \eqref{e:hamiltonian}. We further assume that $H$ is convex with respect to $x$ and $u$ and that $\Phi$ is convex with respect to $x$. Then $\overline u$ is an optimal control for \eqref{state equation}--\eqref{scp}.
		
	\end{theorem}
	\begin{proof}
	It suffices to show 	that $J(u)-J(\overline u)\geq 0$ holds for  all $u\in \mathbf U[0,T]$.		From the convexity of $\Phi$, we have
		\begin{align}
			&J(u)-J(\overline u)\notag\\
		=&\E\int_0^T\big[ f(t,x^u(t),u(t))-f(t,\overline x(t),\overline u(t))\big]dt+\E\big[\Phi(x^u(T)-\overline x(T))\big]\notag\\\geq& \E\int_0^T \big[f(t,x^u(t),u(t))-f(t,\overline x(t),\overline u(t))\big]dt+\E\big[\Phi_x(\overline x(T))\big(x^u(T)-\overline x(T)\big)\big].\label{e:J-J}
		\end{align}
		Applying It\^o's formula to $\la y(t),x^u(t)-\overline x(t)\ra$ and then taking expectation, we obtain
		\begin{align}			&\E\big[\la\Phi_x(\overline x(T)),x^u(T)-\overline x(T)\ra\big]\notag\\ 
			=& \E\int_0^T\la-b^*_{x}(t) y(t)-\left(\frac{\partial}{\partial x}\tr[z(t) q^*(t,\overline{x}(t),\overline{u}(t),\overline{x}(t),\overline{u}(t))]\right)^*-f^*_x(t),x^u(t)-\overline x(t)\ra dt\notag\\&+\E\int_0^T\la y(t),b(t,x^u(t),u(t))-b(t,\overline x(t),\overline u(t))\ra dt\notag\\&+\E \la \int_0^\cdot z(t)M(dt,\overline x(t),\overline u(t))+N(\cdot),\int_0^\cdot\Big( M(dt,x^u(t),u(t))-M(dt,\overline x(t),\overline u(t))\Big)\ra_T.\label{e:3-28}
		\end{align}
		Hence, 	by combining \eqref{e:J-J} and \eqref{e:3-28}, and  	using the expression \eqref{e:hamiltonian} of $H$, we have 
		\begin{align*}
			J(u)-J(\overline u)&\geq \E\int_0^T\Big\{ -H_x\big(t,\overline x(t),\overline u(t),y(t),z(t)\big)\big(x^u(t)-\overline x(t)\big)\\&\hspace{8em}+H\big(t,x^u(t),u(t),y(t),z(t)\big)-H\big(t,\overline x(t),\overline u(t),y(t),z(t)\big)\Big\}dt.		\end{align*}
		It follows from the convexity of $H$ that
		\begin{align*}
			&H\big(t,x^u(t),u(t),y(t),z(t)\big)-H\big(t,\overline x(t),\overline u(t),y(t),z(t)\big)\\&\hspace{1em}\geq H_x\big(t,\overline x(t),\overline u(t),y(t),z(t)\big)\big(x^u(t)-\overline x(t)\big)\\&\hspace{1.5em}+ H_u\big(t,\overline x(t),\overline u(t),y(t),z(t)\big)\big(u(t)-\overline u(t)\big).
		\end{align*}
		Therefore, 
		\begin{align*}
			J(u)-J(\overline u)&\geq E\int_0^T  H_u\big(t,\overline x(t),\overline u(t),y(t),z(t)\big)\big(u(t)-\overline u(t)\big) dt\geq0.
		\end{align*}
		where the last step  follows from (\ref{ine}). The proof is concluded. 
	\end{proof}

\section{A discussion on the stochastic LQ problem}\label{sec:LQ}
In this section, we discuss the stochastic linear quadratic problem (LQ problem) in our setting, where the controlled system (\ref{state equation}) is driven by a local martingale $M(t,x,u)$ 	with $(x,u)$ as  parameters. To make \eqref{state equation} ``linear'' in terms of $(x,u)$ in the martingale part, we impose the following condition on the local characteristic $q$ of $M$: for any $d\times d$ matrix $A$ and all $(x,u), (y,v)\in \R^{d+k}$,
\begin{align}\label{e:condition-q}
&\tr\big[A\big(q^*(t,x,u,y,v)-q^*(t,x,u,x,u)\big)\big]\notag\\
=&\la\left(\frac{\partial}{\partial x}\tr\big[Aq^*(t,x,u,x,u)\big]\right)^*,y-x\ra+\la\left(\frac{\partial}{\partial u}\tr\big[Aq^*(t,x,u,x,u)\big]\right)^*,v-u\ra.
\end{align}

	Now, we consider 	the following linear state equation,
\begin{equation}\label{state equation2}
\left\{ 
\begin{array}{rl}
dx^u(t)= & \left[A(t)x^u(t)+B(t)u(t)\right]dt+M\left(dt,
x^u(t),u(t)\right),  \\ 
x^u\left( 0\right) = & x^u_{0},
\end{array}
\right. 
\end{equation}
with the quadratic cost functional
\begin{equation}\label{scp'}
J(u) =\frac{1}{2}\E\left\{\int_0^T\Big[\la Q(t)x^u(t),x^u(t)\ra+\la R(t)u(t),u(t)\ra\Big] dt+\la Gx^u(T),x^u(T)\ra\right\}.
\end{equation}
	
Here, for $t\in[0,T]$, $A(t)$,	 and $B(t)$ are matrices  with appropriate dimensions,  $Q(t)$ and $G$ are symmetric nonnegative definite matrices, and $R(t)$ is a symmetric positive definite matrix. Here we  use    $\mathbf U[0,T]=\mathcal{M}^2\big([0,T];\R^k\big)$  to denote the set of admissible controls. Then{\color{blue},} the  adjoint equation \eqref{adjoint} becomes
\begin{equation}
\left\{ 
\begin{array}{rl}\label{adjoint2}
dy(t)= & -\big(A^*(t)y(t)+\left(\frac{\partial}{\partial x}\tr\left[z(t)q^*\big(t,\overline x(t),\overline u(t),\overline x(t),\overline u(t)\big)\right]\right)^*+Q(t)\overline x(t)\big  )dt\\&+z(t)dM(t)+dN(t), \\ 
y(T)= & G\overline{x}(T).
\end{array}
\right. 
\end{equation}
	Now, the Hamiltonian \eqref{e:hamiltonian} is
\begin{align*}
H\big(t,x,u,y,z\big)&=\la A(t)x(t)+B(t)u(t),y(t)\ra+\tr\left[z(t) q^*\big(t,x(t),u(t),x(t),u(t)\big)\right]\\&\hspace{1.2em}+\frac{1}{2}\la Q(t)x(t),x(t)\ra+\frac{1}{2}\la R(t)u(t),u(t)\ra+\frac{1}{2}\la G(t)x(T),x(T)\ra.
\end{align*}
	Then, it  follows from the stochastic maximum principle (Theorem \ref{mp} and Remark \ref{remark:3-2}) that
\begin{align}\label{necessary}
 B^*(t)y(t)+\left(\frac{\partial}{\partial u}\tr\left[z(t) q^*\big(t,\overline{x}(t),\overline{u}(t),\overline{x}(t),\overline{u}(t)\big)\right]\right)^*+R(t)\overline{u}(t)= 0 
\end{align}
holds for a.e. $t\in[0,T]$  almost surely, which is a necessary condition for an optimal pair $(\overline x, \overline u)$. As in the classical situation, now we verify that $\overline{u}$ satisfying the necessary condition (\ref{necessary}) is actually an optimal control for the generalized stochastic LQ problems.
\begin{theorem}\label{lq}
If  $\overline{u}$ satisfies \eqref{necessary}, then $\overline{u}$ is an  optimal control for the generalized linear quadratic problem \eqref{state equation2}--\eqref{scp'}.
\end{theorem}
\begin{proof}
	To prove the optimality of $\overline u$, it suffices to show  $J(u)-J(\overline u)\geq 0$ for all $u\in\mathbf U[0,T]$. 	From the nonnegative definiteness of $Q(t), R(t)$ and $G$, we have
	\begin{align}\label{22}
	\notag&J(u)-J(\overline{u})\\\notag=&\frac{1}{2}\E\Bigg\{\int_0^T\Big[\la Q (t)x^u(t),x^u(t)\ra-\la Q(t)\overline{x}(t),\overline{x}(t)\ra+\la R(t)u(t),u(t)\ra-\la R(t)\overline{u}(t),\overline{u}(t)\ra\Big] dt\\\notag&\hspace{5em}+\la Gx^u(T),x^u(T)\ra-\la G\overline{x}(T),\overline{x}(T)\ra\Bigg\}\\\geq& \E\Bigg\{\int_0^T\Big[\la Q(t)\overline{x}(t),x^u(t)-\overline{x}(t)\ra+\la R(t)\overline{u}(t),u(t)-\overline{u}(t)\ra\Big] dt+\la G\overline{x}(T),x^u(T)-\overline{x}(T)\ra\Bigg\}.
	\end{align}
		Then, applying It\^o's formula to $\la x^u(t)-\overline{x}(t),y(t)\ra$, we have
	\begin{align}
	&\E\la G\overline x(T),x^u(T)-\overline x(T)\ra\notag\\=& \E\int_0^T\la-A^*(t)y(t)-\left(\frac{\partial}{\partial x}\tr\left[z(t) q^*\big(t,\overline x(t),\overline u(t),\overline x(t),\overline u(t)\big)\right]\right)^*-Q(t)\overline x(t),x^u(t)-\overline x(t)\ra dt\notag\\&\hspace{1em}+\E\int_0^T\la y(t),A(t)\big(x^u(t)-\overline x(t)\big)+B(t)\big(u(t)-\overline u(t)\big)\ra dt\notag\\&\hspace{1em}+\E\int_0^Td \la \int_0^\cdot z(s)M(ds)+\int_0^\cdot dN(s),\int_0^\cdot M\big(ds,x^u(s),u(s)\big)-\int_0^\cdot M\big(ds,\overline x(s),\overline u(s)\big)\ra_t\notag\\
	=&\E\int_0^T\Bigg[\la -Q(t)\overline{x}(t),x^u(t)-\overline{x}(t)\ra+\la B^*(t)y(t),u(t)-\overline{u}(t)\ra \notag\\&\hspace{5em}+\la -\left(\frac{\partial}{\partial x}\tr\left[z(t) q^*\big(t,\overline x(t),\overline u(t),\overline x(t),\overline u(t)\big)\right]\right)^*,x^u(t)-\overline x(t)\ra\Bigg] dt\notag \\&\hspace{5em}+\E\int_0^T\tr\Big[z(t)\big(q^*(t,x^u(t),u(t),\overline x(t),\overline u(t))-q^*(t,\overline x(t),\overline u(t),\overline x(t),\overline u(t))\big)\Big]dt
	\notag\\
	=&\E\int_0^T\Bigg[\la -Q(t)\overline{x}(t),x^u(t)-\overline{x}(t)\ra\notag\\
	&\qquad \qquad +\la B^*(t)y(t)+\left(\frac{\partial}{\partial u}\tr\left[z(t)q\big(t,\overline x(t),\overline u(t),\overline x(t),\overline u(t)\big)\right]\right)^*,u(t)-\overline{u}(t)\ra\Bigg] dt.\label{e:4.7}
	\end{align}
	where the last equality follows from \eqref{e:condition-q}. Then the desired inequality follows from (\ref{necessary}), (\ref{22}) and \eqref{e:4.7}:
	\begin{align*}
	&J(u )-J(\overline{u})\\ \geq &\E\int_0^T \la R(t)\overline{u}(t)+B^*(x)y(t)+\left(\frac{\partial}{\partial u}\tr\left[z(t)q\big(t,\overline x(t),\overline u(t),\overline x(t),\overline u(t)\big)\right]\right)^*,u(t)-\overline{u}(t)\ra dt\\=&0.
	\end{align*}
	This concludes the proof.
\end{proof}
\begin{remark}
It can be easily checked that if $q$ is linear with respect to $x$, $ y$, $u$ and $v$, then  condition \eqref{e:condition-q} is fulfilled, 	and hence  the classical LQ problem is recovered.  	More precisely, 	we consider the linear form of \eqref{classical-2}:
	\begin{equation}\label{lcsde}
	dx^u(t)=x^u_0+\big[A(t)x^u(t)+B(t)u(t)\big]dt+\sum\limits_{j=1}^m\big[C_j(t)x^u(t)+D_j(t)u(t)\big]dW^j_t,
	\end{equation}
where $A$, $B$, $C_j$, $D_j$ are deterministic matrix-valued functions of suitable dimensions.	Then, the local characteristic $q$ of \[M(t,x,u) = \sum_{j=1}^m \left[\int_0^t C_j(s) x dW^j_s+ \int_0^t D_j(s) udW^j_s\right]\] is given by 
\begin{align*}
	q(t,x,u,y,v)
	=\sum\limits_{j=1}^m\Big[C_j(t)xy^*C_j^*(t)+C_j(t)xv^*D^*_j(t)+D_j(t)uy^*C^*_j(t)+D_j(t)uv^*D^*_j(t)\Big],
\end{align*}
which satisfies \eqref{e:condition-q} with equality.\\

	 In this situation, condition \eqref{necessary} can be written as   \[B^*(t)y(t)+\sum\limits_{j=1}^m\tr\big[z(t)D^*_j(t)\big(C_j(t)\overline x(t)+D_j(t)\overline u(t)\big)\big]+R(t)\overline u(t)=0.\]
	This is consistent with the variational inequality \eqref{va-c} (which is an equality when $U=\R^k$ by Remark \ref{remark:3-2}) in the classical setting. 
\end{remark}

\medskip

{\bf Acknowledgement.} The authors would like to thank Prof. Mingshang Hu for 	his helpful discussions.  	The authors are also 	grateful to the two anonymous referees for their valuable comments.  J. Song is partially supported by Shandong University grant 11140089963041 and 	the National Natural Science Foundation of China grant 12071256.


\bibliographystyle{plain}
\bibliography{Reference}

\begin{bibdiv}
\begin{biblist}

\bib{Peter}{article}{
      author={Bank, Peter},
      author={Baum, Dietmar},
       title={Hedging and portfolio optimization in financial markets with a
  large trader},
        date={2004},
        ISSN={0960-1627},
     journal={Math. Finance},
      volume={14},
      number={1},
       pages={1\ndash 18},
         url={https://doi.org/10.1111/j.0960-1627.2004.00179.x},
      review={\MR{2030833}},
}

\bib{Bismut78}{article}{
      author={Bismut, Jean-Michel},
       title={An introductory approach to duality in optimal stochastic
  control},
        date={1978},
        ISSN={0036-1445},
     journal={SIAM Rev.},
      volume={20},
      number={1},
       pages={62\ndash 78},
         url={https://doi.org/10.1137/1020004},
      review={\MR{469466}},
}

\bib{li2016}{article}{
      author={Buckdahn, Rainer},
      author={Li, Juan},
      author={Ma, Jin},
       title={A stochastic maximum principle for general mean-field systems},
        date={2016},
        ISSN={0095-4616},
     journal={Appl. Math. Optim.},
      volume={74},
      number={3},
       pages={507\ndash 534},
         url={https://doi.org/10.1007/s00245-016-9394-9},
      review={\MR{3575614}},
}

\bib{ChenEpstein2002}{article}{
      author={Chen, Zengjing},
      author={Epstein, Larry},
       title={Ambiguity, risk, and asset returns in continuous time},
        date={2002},
        ISSN={0012-9682},
     journal={Econometrica},
      volume={70},
      number={4},
       pages={1403\ndash 1443},
         url={https://doi.org/10.1111/1468-0262.00337},
      review={\MR{1929974}},
}

\bib{dumeng2013}{article}{
      author={Du, Kai},
      author={Meng, Qingxin},
       title={A maximum principle for optimal control of stochastic evolution
  equations},
        date={2013},
     journal={SIAM Journal on Control and Optimization},
      volume={51},
      number={6},
       pages={4343\ndash 4362},
}

\bib{El94}{incollection}{
      author={El~Karoui, N.},
      author={Huang, S.-J.},
       title={A general result of existence and uniqueness of backward
  stochastic differential equations},
        date={1997},
   booktitle={Backward stochastic differential equations ({P}aris,
  1995--1996)},
      series={Pitman Res. Notes Math. Ser.},
      volume={364},
   publisher={Longman, Harlow},
       pages={27\ndash 36},
      review={\MR{1752673}},
}

\bib{fabbri2017stochastic}{article}{
      author={Fabbri, Giorgio},
      author={Gozzi, Fausto},
      author={Swiech, Andrzej},
       title={Stochastic optimal control in infinite dimension},
        date={2017},
     journal={Probability and Stochastic Modelling. Springer},
}

\bib{fuhrman2013}{article}{
      author={Fuhrman, Marco},
      author={Hu, Ying},
      author={Tessitore, Gianmario},
       title={Stochastic maximum principle for optimal control of {SPDE}s},
        date={2013},
     journal={Applied Mathematics \& Optimization},
      volume={68},
      number={2},
       pages={181\ndash 217},
}

\bib{HanPengWu}{article}{
      author={Han, Yuecai},
      author={Peng, Shige},
      author={Wu, Zhen},
       title={Maximum principle for backward doubly stochastic control systems
  with applications},
        date={2010},
        ISSN={0363-0129},
     journal={SIAM J. Control Optim.},
      volume={48},
      number={7},
       pages={4224\ndash 4241},
         url={https://doi.org/10.1137/080743561},
      review={\MR{2665464}},
}

\bib{Hu2017}{article}{
      author={Hu, Mingshang},
       title={Stochastic global maximum principle for optimization with
  recursive utilities},
        date={2017},
     journal={Probab. Uncertain. Quant. Risk},
      volume={2},
       pages={Paper No. 1, 20},
         url={https://doi.org/10.1186/s41546-017-0014-7},
      review={\MR{3625730}},
}

\bib{Hu2018}{article}{
      author={Hu, Mingshang},
      author={Ji, Shaolin},
      author={Xue, Xiaole},
       title={A global stochastic maximum principle for fully coupled
  forward-backward stochastic systems},
        date={2018},
        ISSN={0363-0129},
     journal={SIAM J. Control Optim.},
      volume={56},
      number={6},
       pages={4309\ndash 4335},
         url={https://doi.org/10.1137/18M1179547},
      review={\MR{3881235}},
}

\bib{HuYPeng1990}{article}{
      author={Hu, Ying},
      author={Peng, Shi~Ge},
       title={Maximum principle for semilinear stochastic evolution control
  systems},
        date={1990},
        ISSN={1045-1129},
     journal={Stochastics Stochastics Rep.},
      volume={33},
      number={3-4},
       pages={159\ndash 180},
         url={https://doi.org/10.1080/17442509008833671},
      review={\MR{1082339}},
}

\bib{JiZhou2006}{article}{
      author={Ji, Shaolin},
      author={Zhou, Xun~Yu},
       title={A maximum principle for stochastic optimal control with terminal
  state constraints, and its applications},
        date={2006},
     journal={Commun. Inf. Syst.},
      volume={6},
      number={4},
       pages={321\ndash 338},
}

\bib{Kunita86}{book}{
      author={Kunita, H.},
       title={Lectures on stochastic flows and applications},
      series={Tata Institute of Fundamental Research Lectures on Mathematics
  and Physics},
   publisher={Published for the Tata Institute of Fundamental Research, Bombay;
  by Springer-Verlag, Berlin},
        date={1986},
      volume={78},
        ISBN={3-540-12878-6},
      review={\MR{867686}},
}

\bib{Kunita90}{book}{
      author={Kunita, Hiroshi},
       title={Stochastic flows and stochastic differential equations},
      series={Cambridge Studies in Advanced Mathematics},
   publisher={Cambridge University Press, Cambridge},
        date={1990},
      volume={24},
        ISBN={0-521-35050-6},
      review={\MR{1070361}},
}

\bib{LiJuan2013}{article}{
      author={Li, Juan},
       title={Stochastic maximum principle in the mean-field controls},
        date={2012},
        ISSN={0005-1098},
     journal={Automatica J. IFAC},
      volume={48},
      number={2},
       pages={366\ndash 373},
         url={https://doi.org/10.1016/j.automatica.2011.11.006},
      review={\MR{2889429}},
}

\bib{Liang2007}{article}{
      author={Liang, Zongxia},
       title={Stochastic differential equations driven by spatial parameters
  semimartingale with non-{L}ipschitz local characteristic},
        date={2007},
        ISSN={0926-2601},
     journal={Potential Anal.},
      volume={26},
      number={4},
       pages={307\ndash 322},
         url={https://doi.org/10.1007/s11118-007-9038-4},
      review={\MR{2300335}},
}

\bib{lu2014general}{book}{
      author={L{\"u}, Qi},
      author={Zhang, Xu},
       title={General {P}ontryagin-type stochastic maximum principle and
  backward stochastic evolution equations in infinite dimensions},
   publisher={Springer},
        date={2014},
}

\bib{MaJ}{book}{
      author={Ma, Jin},
      author={Yong, Jiongmin},
       title={Forward-backward stochastic differential equations and their
  applications},
      series={Lecture Notes in Mathematics},
   publisher={Springer-Verlag, Berlin},
        date={1999},
      volume={1702},
        ISBN={3-540-65960-9},
      review={\MR{1704232}},
}

\bib{Matoussi}{article}{
      author={Matoussi, Anis},
      author={Scheutzow, Michael},
       title={Stochastic {PDE}s driven by nonlinear noise and backward doubly
  {SDE}s},
        date={2002},
        ISSN={0894-9840},
     journal={J. Theoret. Probab.},
      volume={15},
      number={1},
       pages={1\ndash 39},
         url={https://doi.org/10.1023/A:1013803104760},
      review={\MR{1883201}},
}

\bib{16}{article}{
      author={Meyer-Brandis, Thilo},
      author={\O~ksendal, Bernt},
      author={Zhou, Xun~Yu},
       title={A mean-field stochastic maximum principle via {M}alliavin
  calculus},
        date={2012},
        ISSN={1744-2508},
     journal={Stochastics},
      volume={84},
      number={5-6},
       pages={643\ndash 666},
         url={https://doi.org/10.1080/17442508.2011.651619},
      review={\MR{2995516}},
}

\bib{Global}{article}{
      author={Peng, Shi~Ge},
       title={A general stochastic maximum principle for optimal control
  problems},
        date={1990},
        ISSN={0363-0129},
     journal={SIAM J. Control Optim.},
      volume={28},
      number={4},
       pages={966\ndash 979},
         url={https://doi.org/10.1137/0328054},
      review={\MR{1051633}},
}

\bib{SSZ19}{article}{
      author={Song, Jian},
      author={Song, Xiaoming},
      author={Zhang, Qi},
       title={Nonlinear {F}eynman--{K}ac formulas for stochastic partial
  differential equations with space-time noise},
        date={2019},
     journal={SIAM Journal on Mathematical Analysis},
      volume={51},
      number={2},
       pages={955\ndash 990},
}

\bib{Tang}{article}{
      author={Tang, Shanjian},
       title={The maximum principle for partially observed optimal control of
  stochastic differential equations},
        date={1998},
        ISSN={0363-0129},
     journal={SIAM J. Control Optim.},
      volume={36},
      number={5},
       pages={1596\ndash 1617},
         url={https://doi.org/10.1137/S0363012996313100},
      review={\MR{1626880}},
}

\bib{Wuzhen2013}{article}{
      author={Wu, Zhen},
       title={A general maximum principle for optimal control of
  forward-backward stochastic systems},
        date={2013},
        ISSN={0005-1098},
     journal={Automatica J. IFAC},
      volume={49},
      number={5},
       pages={1473\ndash 1480},
         url={https://doi.org/10.1016/j.automatica.2013.02.005},
      review={\MR{3044030}},
}

\bib{YZ99}{book}{
      author={Yong, Jiongmin},
      author={Zhou, Xun~Yu},
       title={Stochastic controls: Hamiltonian systems and {HJB} equations},
   publisher={Springer Science \& Business Media},
        date={1999},
      volume={43},
}

\bib{ZhouXY}{article}{
      author={Zhou, Xun~Yu},
       title={A unified treatment of maximum principle and dynamic programming
  in stochastic controls},
        date={1991},
        ISSN={1045-1129},
     journal={Stochastics Stochastics Rep.},
      volume={36},
      number={3-4},
       pages={137\ndash 161},
      review={\MR{1128491}},
}

\end{biblist}
\end{bibdiv}
\end{document}